\documentclass[12pt,reqno,oneside]{amsart}
\usepackage{amssymb}
\usepackage{amsmath}
\usepackage{mathtools}
\usepackage{amsthm}
\usepackage{mathrsfs}
\usepackage{geometry}
\usepackage[utf8]{inputenc}
\usepackage[T1]{fontenc}
\usepackage{fullpage}

\usepackage{commath}
\usepackage{hyperref}
\usepackage{multirow}
\usepackage{xcolor}
\usepackage{breqn}
\usepackage{tikz}
\usetikzlibrary{matrix}

\input colordvi

\newtheorem{thm}[equation]{Theorem}
\newtheorem{cor}[equation]{Corollary}
\newtheorem{lem}[equation]{Lemma}

\newtheorem{prop}[equation]{Proposition}
\newtheorem{definition}[equation]{Definition}

\numberwithin{equation}{section}

\renewcommand\a{\alpha}
\renewcommand\b{\beta}
\newcommand\g{\gamma}
\renewcommand\d{\delta}

\renewcommand\l{\lambda}

\newcommand\f{\frac}

\newcommand{\Z}{{\mathbb{Z}}}
\newcommand{\R}{{\mathbb{R}}}

\newcommand{\C}{{\mathbb{C}}}

\newcommand\re{\text{Re~}}

\newcommand{\Inv}{\operatorname{Inv}}

\renewcommand\i{^{-1}}

\newcommand{\sgn}{\operatorname{sgn}}
\newcommand{\dn}{\,dn}
\newcommand{\du}{\,du}
\newcommand{\pihat}{\widehat{\pi}}
\newcommand{\phihat}{\widehat{\varphi}}
\newcommand{\what}{\widehat{w}}
\newcommand{\ahat}{\widehat{a}}
\newcommand{\bhat}{\widehat{b}}
\newcommand{\uhat}{\widehat{u}}
\newcommand{\nhat}{\widehat{n}}

\newcommand{\pitilde}{\widetilde{\pi}}
\newcommand{\phitilde}{\widetilde{\varphi}}
\newcommand{\wtilde}{\widetilde{w}}
\newcommand{\ntilde}{\widetilde{n}}
\newcommand{\utilde}{\widetilde{u}}
\newcommand{\atilde}{\widetilde{a}}
\newcommand{\btilde}{\widetilde{b}}
\newcommand{\wntilde}{\widetilde{wn}}
\newcommand{\wutilde}{\widetilde{wu}}

\newcommand{\gobble}[1]{}
  \newcommand{\rangeref}[2]{%
    \ref{#1}--\afterassignment\gobble\fam 0\ref{#2}%
  }

\newenvironment{smallarray}[1]
 {\null\,\vcenter\bgroup\scriptsize
  \arraycolsep=.13885em
  \hbox\bgroup$\array{@{}#1@{}}}
 {\endarray$\egroup\egroup\,\null}

\newcommand{\GL}{\operatorname{GL}}

\newcommand{\npdj}{n_{\pi(d),j}}

\newcommand{\nprj}{n_{\pi(r),j}}

\newcommand{\inR}{\int_{-\infty}^{\infty}}
\newcommand{\dx}{\,dx}
\newcommand{\eqand}{\quad\text{and}\quad}
\newcommand{\beq}{\begin{equation}}
\newcommand{\eeq}{\end{equation}}

\title{Schubert cells and Whittaker functionals for $\GL(r,\R)$ part II: Existence via integration by parts}
\author{Doyon Kim}
\address{Mathematical Institute of the University of Bonn, Bonn, Germany}
\email{kimdoyon@math.uni-bonn.de}
\date{\today}
\begin{document}

\begin{abstract}
We give a new proof of the existence of Whittaker functionals for principal series representation of $\GL(r,\R)$, utilizing the analytic theory of distributions. We realize Whittaker functionals as equivariant distributions on $\GL(r,\R)$, whose restriction to the open Schubert cell is unique up to a constant. Using a birational map on the Schubert cells, we show that the unique distribution on the open Schubert cell extends to a distribution on the entire space $\GL(r,\R)$. This technique gives a proof of the analytic continuation of Jacquet integrals via integration by parts. We briefly discuss an application of the method to the Bessel functions on $\GL(r,\R)$.
\end{abstract}

\maketitle

\tableofcontents

\section{Introduction}
Let $G$ be a real reductive Lie group and $(\pi,V)$ be an irreducible representation of $G$. Whittaker functionals are the continuous linear functionals $\tau : V^\infty \to \C$ on the space of smooth vectors which satisfy
\begin{equation}
\label{whitfunl}
    \tau\left(\pi(n)v\right) = \psi(n)\,\tau(v) \quad \text{for all} \quad n\in N, 
\end{equation}
where $N$ is the unipotent radical of a minimal parabolic subgroup of $G$ and $\psi : N \to \C^*$ is a nondegenerate character. When such a functional exists, it gives rise to a Whittaker function on $G$ by $g\mapsto \tau(\pi(g)v)$. The systematic study of Whittaker functions was initiated by Jacquet \cite{jacquet1967fonctions}. In his thesis \cite{jacquet1967fonctions}, he introduced integral representations of Whittaker functions on Chevalley groups and gave a local proof of their analytic continuation. Schiffmann \cite{schiffmann1971integrales} extended the result to any rank
one real groups. Piatetski-Shapiro \cite{piatetski1979multiplicity} and Shalika \cite{shalika1974multiplicity} proved the ``multiplicity one theorem", which asserts that the space of Whittaker functionals on irreducible representations of $\GL(r,\R)$ is at most one-dimensional. In \cite{kostant1978whittaker}, Kostant showed that the dimension of the space of Whittaker functionals for any principal series representation of a quasisplit linear Lie group is exactly one.
\par
In this paper, we focus on the group $\GL(r,\R)$. This paper is the second in a series of two papers that together give a new proof of the existence of Whittaker models on principal series representations of $\GL(r,\R)$ by reducing the analysis of Whittaker functionals to integration by parts. This series is based on the author's Ph.D. thesis \cite{kimWhit}, which originally presented this new proof. The proofs have been modified and improved in many places for clarity. \par
While the analytic continuation of Jacquet integrals for $\GL(r,\R)$ is a special case of the results mentioned above, and has been proved in different ways several times, for example in \cite{casselman2000bruhat} and \cite{goodman1980whittaker}, our new method has its own merit towards the applications to problems in analytic number theory. We directly analyze Jacquet integrals through a series of integral transforms and avoid the use of any algebraic machinery, making our method well-suited for analytic applications. We give one such example in Section~\ref{sec:Bessel}, where we discuss Bessel functions for $\GL(r,\R)$. It turns out that change of variables we use can be applied not only to Jacquet integrals, but also to the Bessel integrals for $\GL(r,\R)$.\par
The representation theoretic question regarding the existence of Whittaker models can be translated into a question of the existence of certain equivariant distributions. A Whittaker functional for a principal series representation $(\pi,V)$ of $G=\GL(r,\R)$ can be realized as a distribution vector $\tau$ satisfying the transformation law \eqref{whitfunl}, and the transformation law uniquely determines $\tau$ on the open Schubert cell. The existence and uniqueness of the Whittaker functional amounts to the assertion that $\tau$ admits a canonical extension to the whole group $G$, and that this extension becomes unique if it is required to transform under the action of $N$ according to the character $\psi$. Casselman, Hecht, and Mili\v{c}i\'{c} take this point of view in \cite{casselman2000bruhat}, where they establish the existence and uniqueness of Whittaker functionals. \par 
From this viewpoint, the Jacquet integral can be interpreted as a pairing of $\tau$ and a smooth function. We construct an extension of $\tau$ from the open Schubert cell to $G$ via a series of integral transforms and integration by parts. This uses the analytic technique developed by Miller and Schmid in \cite{miller2004distributions}, in particular, the notion of distributions vanishing to infinite order. After the transformation, the Jacquet integrals are simplified up to the point where we can prove their analytic continuation via integration by parts just like the $\Gamma$-function. \par 
This integration-by-parts method is used in \cite{miller2012archimedean} by Miller and Schmid to prove the analytic continuation of the exterior square $L$-functions. The implementation of this method for Whittaker functionals was originally formulated in their unpublished note \cite{miller2008unpublished}, where they posed the primary result of this series, Theorem~\ref{biratlthm}, as a conjecture. See \cite{miller2013mathematics} for an overview of the integration-by-parts method. A similar approach is used in \cite{kim2023infinitely} by the author to prove the infinitude of the critical zeros of $L$-functions with additive twists, where integration by parts is applied to the integral representations of $L$-functions to obtain their upper bounds. 
\subsection{Set-up and main results}
The set-up in this section closely follows \cite{miller2008unpublished}. Let $F$ be a field of characteristic $0$. Let $G=\GL(r,F)$, $B_{-}$ be its lower triangular Borel subgroup, and $N$ (resp. $N_{-}$) is the group of upper triangular (resp. lower triangular) unipotent matrices. The group $N$ can be identified with $F^d$ with $d=\f12 r(r-1)$ via the coordinate entries $n_{i,j}$, $1\le i<j\le r$ of $n\in N$. Let $\Omega$ be the Weyl group of $G$. The Weyl group consists of permutation matrices $w=(w_{i,j})$ with 
\begin{equation}\label{pidef}w_{i,j}=\d_{j=\pi(i)},\end{equation}
in bijection with permutations $\pi$ in the symmetric group $\mathfrak{S}_{r}$. The group $G$ has the Bruhat decomposition
$G=\bigcup_{w\in\Omega} NwB_{-}$, where the union is disjoint. For $w\in \Omega$, the $N$-orbit $NwB_{-}$ of $wB_{-}$ is called the \emph{Schubert cell attached to} $w$. We write $S_w=NwB_{-}$. The Schubert cell $NB_{-}$ attached to the identity element is also called the \emph{open Schubert cell}. For each $w\in\Omega$ we may decompose $N$ as the product of the subgroups
\begin{equation}\label{nws}
    \aligned
    N^w & = \{n\in N\,|\,w\i n w\in N_{-}\} =  N\cap wN_{-}w\i,\\
    N_w & = \{n\in N\,|\,w\i n w\in N\} =   N \cap wNw\i = 
    w N_{w\i}w\i.
    \endaligned
\end{equation}
It is clear that for each $w\in\Omega$, we have $N=N_wN^w  = N^wN_w$.
\begin{lem}[Lemma 1.4 of \cite{kim2024pt1}] \label{wNB} 
We have $S_w=N_w w B_{-}$, hence $S_w$ is a left $N_w$-orbit of $wB_{-}$. The Schubert cell $S_w$ is contained in the set $wNB_{-}$, which satisfies
\[wNB_{-}= N_w w N^{w\i}B_{-}.\]
\end{lem}
Write $Y_w=wNB_{-}$. We can write $G$ as an overlapping union $\bigcup_{w\in\Omega}Y_w$. Each $Y_w$ has a large overlap with the open Schubert cell $NB_{-}$, and the overlap is dense and open in $Y_w$. Writing $g\in Y_w$ as $g=n_1\cdot w\cdot n\cdot b$ with $n_1 \in N_w$, $n\in N^{w\i}$, and $b\in B_{-}$, we see that $g$ lies in the open Schubert cell if and only if the matrix $wn\in wN^{w\i}$ admits a $UDL$ decomposition. \par
We shall assign a set of coordinates on $wN^{w\i}B_-$ that describes the entries of $wn$.
Let $e_{i,j}$ be an elementary matrix with $1$ in the $(i,j)$-entry and zeros in all other entries. Then a Weyl group element $w\in \Omega$ and the corresponding permutation $\pi\in \mathfrak{S}_{r}$ satisfy $w e_{i,j} w\i=e_{\pi\i(i),\pi\i(j)}$.
By \eqref{nws}, the entries $n_{i,j}$ of $n=(n_{i,j})\in N^{w\i}$ above the main diagonal are free if $\pi\i(i)>\pi\i(j)$, and zero otherwise. Let
\begin{equation}\label{invpiinv}
    \Inv(\pi\i)  =  \{ (i,j)\mid i<j, \ \pi\i(i) > \pi\i(j) \}
\end{equation}
denote the set of indices of free entries in $n\in N^{w\i}$. If a matrix $X$ has entries $x_{i,j}$, then the entries of $wX$ are
    $(wX)_{i,j} = \sum_{k=1}^n w_{i,k}x_{k,j} = x_{\pi(i),j}$. Accordingly, we denote the entries of $wn$ as
\begin{equation}\label{wnentries}
(wn)_{i,j}=n_{\pi(i),j}, \quad \text{or equivalently,} \quad n_{i,j}=(wn)_{\pi\i(i),j}.
\end{equation}
By (\ref{invpiinv}), the entries $n_{i,j}$ of $wn$ are free if $(i,j)\in \Inv(\pi\i)$. We call these entries the \emph{free variables of} $wn$. By (\ref{pidef}), $n_{i,j}=1$ if $i=j$. Finally, $n_{i,j}=0$ if $i>j$ or $\pi\i(i)<\pi\i(j)$. We denote the set of free variables as $V_w$, so
\begin{equation}\label{coords}
    V_w=\{n_{i,j}\,|\,(i,j)\in \Inv(\pi\i)\}.
\end{equation}
We often use the notation $\a=(i,j)$ for $(i,j)\in \Inv(\pi\i)$. \par 
We assign an order to the indices $(i,j)\in \Inv(\pi\i)$ using the lexicographic order of the pair $\left(-j,-\pi\i(i)\right)$. That is, 
\begin{equation}\label{ordering}
  (i_1,j_1)\succ (i_2,j_2) \quad \text{if}\quad j_1<j_2, \quad \text{or if} \quad j_1=j_2 \quad \text{and}\quad  \pi\i(i_1)<\pi\i(i_2).
\end{equation}
We extend the ordering in the obvious way to $n_{i,j}\in V_w$, that is, $n_{i_1,j_1}\succ n_{i_2,j_2}$ if $(i_1,j_1)\succ (i_2,j_2)$. With this ordering, $n_\a \succ n_\b$ if $n_\a$ is either to the left of $n_\b$ in the matrix $wn$, or if $n_\a$ and $n_\b$ are in the same column of $wn$ but $n_\a$ is above $n_\b$ in the matrix $wn$. Our main result concerns a birational map on $wN^{w\i}B_{-}$.
\begin{thm}\label{biratlthm}
For each $w\in\Omega$, there exists a birational map $R=R_w$
\[V_w \simeq F^{\, \abs{V_w}} \ \stackrel{R}{ \longrightarrow} \{u_{i,j}\mid (i,j)\in \Inv(\pi\i)\}\simeq F^{\,\abs{V_w}} \]
which satisfies the following properties:
\begin{enumerate}
\item[(i)] $R$ is smooth, of maximal rank, on $(F^\times)^{\abs{V_w}}$.
\item[(ii)] Let $wu$ denote the matrix obtained from $wn$ by replacing $n_\a\in V_w$ with $u_\a$. 
If each $n_{i,j}\neq 0$, then $wu$ decomposes as $x\cdot b$, where $x
\in N$ and $b\in B_{-}$. The $i$-th superdiagonal entry $x_{i,i+1}$ of $x$ is given by
$x_{i,i+1}=\sum\limits_{n_{\alpha}\in B_w(i+1)} \frac{1}{n_{\alpha}}$, where $B_w(i)\subseteq V_w$ with $2\leq i\leq r$ is a partition of $V_w$ defined in Definition~\ref{def:bwi}. Thus we have
$\sum\limits_{i=1}^{r-1} x_{i,i+1}=\sum\limits_{n_\alpha \in V_w} \frac{1}{n_\alpha}$.
\item[(iii)] The $i$-th diagonal entry $b_{i,i}$ of $b$ is given by
\begin{equation}\label{bdiag}
    b_{i,i}=(-1)^{i+\pi\i(i)}\cdot\frac{\prod_{(a,i)\in \Inv(\pi\i)}(-n_{a,i})}{\prod_{(i,b)\in \Inv(\pi\i)}n_{i,b}}.
\end{equation}
  \item[(iv)] The measure $\prod\limits_{(i,j)\in \Inv(\pi\i)}du_{i,j}$
    transforms to $ \prod\limits_{(i,j)\in \Inv(\pi\i)} \lvert n_{i,j}\rvert^{j-i-1}dn_{i,j}$.
  \item[(v)] For any pair of indices $\a$ and $\b$, $n_\a\f{\partial }{\partial n_\a} u_\b$ has the form \begin{equation}\label{ordform}
 \sum\limits_{i} c_i(\a)f_i, \quad \text{where}\quad c_i(\a)\in \Z\bigl[\{n_\d,n_\d\i\mid \d\prec\a\}\bigr]\quad\text{and}\quad f_i\in \Z\bigl[\{u_\g \mid n_\g \in V_w \} \bigr].
\end{equation}
\end{enumerate}
\end{thm}
This theorem implies, through an argument using the notion of distributions vanishing to infinite order, the existence of Whittaker functionals for principal series representations. This result is formulated in Theorem~\ref{thm:existence} below, and we deduce it from Theorem~\ref{biratlthm} in Section~\ref{sec:whitdist}. \par
In \cite{miller2008unpublished}, Miller and Schmid obtained an explicit formula for $R_{w_l}$ that satisfies parts (i) through (v), where $w_l$ is the longest Weyl group element of $\GL(r)$ for any $r\geq 2$. Also, they provided an algorithm that computes $R_w$ for a given Weyl group element (see \cite[Appendix B]{kim2024pt1} for the algorithm) and conjectured that the formula obtained by the algorithm satisfies Theorem~\ref{biratlthm}. Using Mathematica, it is verified that their formula for $R_w$ for each Weyl group element $w$ of $\GL(r)$ with $2\leq r\leq 7$ satisfies parts (ii) and (iii). However, they were unable to prove general statements using the algorithmic approach. 
\par
The paper is organized as follows. In Section~\ref{sec:whitdist}, we introduce the notion of Whittaker distributions and explain the application of Theorem~\ref{biratlthm} to Jacquet integrals. Then we briefly discuss applications of the birational map $R_w$ to Bessel functions. In Section~\ref{sec:Rformula}, we provide the explicit formula for $R_w$. In the first paper \cite{kim2024pt1} of this series, several auxiliary results are established by analyzing the combinatorial properties of the map $R_w$. From Sections~\ref{sec:detsub} through~\ref{sec:parts23case2}, we prove parts (ii) and (iii) of Theorem~\ref{biratlthm} using the auxiliary results. After that, we prove the rest of Theorem~\ref{biratlthm}, also using the results from \cite{kim2024pt1}.

\section*{Acknowledgements}
This problem started from the unfinished work of Stephen D. Miller and Wilfried Schmid \cite{miller2008unpublished}. I would like to thank Stephen D. Miller for suggesting the problem and for his valuable advice. I also would like to thank Valentin Blomer for helpful suggestions, and Jack Buttcane for helpful correspondence in which he suggested a direction related to Bessel functions. This work was partially supported by ERC Advanced Grant 101054336 and Germany’s 
Excellence Strategy grant EXC-2047/1 - 390685813.

\section{Whittaker distributions on $\GL(r,\R)$}\label{sec:whitdist}
Let $G=\GL(r,\R)$. For $\lambda=(\lambda_1,\ldots,\lambda_r)\in \mathbb{C}^r$ and $\delta=(\delta_1,\ldots\delta_r)\in \left(\Z/2\Z\right)^r$, define the character $\chi_{\lambda,\delta}:B_{-}\to \mathbb{C}^*$ by the formula
$\chi_{\lambda,\delta}(b)=\prod_{j=1}^r \sgn(b_{j,j})^{\delta_j}\abs{b_{j,j}}^{\lambda_j}$. Let $V_{\lambda,\delta}$ denote the principal series representation induced from the character $\chi_{\lambda-\rho,\delta}$, where 
$\rho=(\frac{r-1}{2},\frac{r-3}{2},\ldots,\frac{1-r}{2})$ is the half sum of the positive roots. That is,
\[
    V_{\lambda,\delta}=\{f\in L^2_{\text{loc}}(G)  \mid  f(gb)=\chi^{-1}_{\lambda-\rho,\delta}(b)f(g) \text{ for } g\in G,\  b\in B_{-}\}.
\]
Let $V_{\lambda,\delta}^\infty$ be the subspace of smooth vectors consisting of $f\in C^{\infty}(G)$. We also define the distributional completion,
\[
    V_{\lambda,\delta}^{-\infty}=\{\sigma\in C^{-\infty}(G) \mid \sigma(gb)=\chi^{-1}_{\lambda-\rho,\delta}(b)\sigma(g) \text{ for } g\in G,\  b\in B_{-}\}.
\]
The group $G$ acts on each of the three spaces $V_{\lambda,\delta}$, $V_{\lambda,\delta}^\infty$, and $V_{\lambda,\delta}^{-\infty}$ via left translation. For $f_1\in  V_{\lambda,\delta}$ and $f_2\in  V_{-\lambda,\delta}$, we have the pairing defined by
\[
   \langle f_1, f_2 \rangle_{\lambda,\delta}=\int_N f_1(n)f_2(n)\,dn.
\]
The pairing extends on the left and restricts on the right to $V_{\lambda,\delta}^{-\infty} \times V_{-\lambda,\delta}^{\infty}$ since distributions are given locally by integration against derivatives of locally $L^2$ functions. Via this pairing, $V_{\lambda,\delta}^{-\infty}$ can be identified with the dual of $V_{-\lambda,\delta}^{\infty}$. See \cite{bate2013metaplectic} and \cite{miller2006automorphic} for more details on pairings of distributions.
\subsection{Whittaker distributions and Jacquet integrals}
Let $e(z)=e^{2\pi i z}$, and let $\psi(n)=e(n_{1,2}+n_{2,3}+\cdots+n_{r-1,r})$ be the standard Whittaker character on $N$. We say that $\tau\in V_{\lambda,\delta}^{-\infty}$ is a Whittaker distribution if $\tau$ satisfies the transformation law
\begin{equation}\label{Whitlaw}
\tau(ngb)=\psi(n)\tau(g)\chi^{-1}_{\lambda-\rho,\delta}(b)
 \end{equation}
for all $n\in N$, $g\in G$, and $b\in B_{-}$. Observe that the transformation law uniquely determines the restriction of $\tau$ to the open Schubert set $NB_{-}$ as a smooth function $\psi(n)\chi^{-1}_{\lambda-\rho,\delta}(b)$ up to a constant. If $\tau\in V_{\lambda,\delta}^{-\infty}$ is a Whittaker distribution, then for any $f\in V_{-\lambda,\delta}^\infty$ the pairing 
\begin{equation}\label{jacquet}
    \langle \tau,f\rangle_{\lambda,\delta}=\int_N \psi(n) f(n) \dn
\end{equation}
is called the Jacquet integral. The function $W_{\psi,f}$ defined by $W_{\psi,f}(g)=\langle \tau, \pi(g)f\rangle_{\lambda,\delta}$ is a Whittaker function, and conversely all smooth Whittaker functions for principal series arise this way. It is well-known that there is an open subset $U$ of $\C^r$ such that the Jacquet integral converges absolutely for $\lambda\in U$ (see \cite[Section 15.4]{wallachII}), and the existence of Whittaker distributions amounts to showing that we can analytically continue the integral to the remaining values. 
\begin{thm}\label{thm:existence}
    The Jacquet integral $\lambda\mapsto \langle \tau,\cdot\rangle_{\lambda,\delta}$ admits an analytic continuation to $\C^r$, and defines a Whittaker functional that depends holomorphically on $\lambda$.
\end{thm}
The proof is given in Section~\ref{sec:jacquetr}. The proof uses the notion of distributions vanishing to infinite order from \cite{miller2004distributions}.
\begin{definition}[Definition 2.4 of \cite{miller2004distributions}]\label{inforder}
For $M$ a smooth manifold and $S\subset M$ a closed submanifold, we say that a distribution $\sigma\in C^{-\infty}(M)$ vanishes to order $k\geq 0$ along $S$ if every point $p\in S$ has an open neighborhood $U_p$ in $M$ with the following property: there exist differential operators $D_j$ on $U_p$ which are tangential to $S\cap U_p$, measurable locally bounded functions $h_j\in L^\infty_{\text{loc}}(U_p)$, and smooth functions $f_j\in C^{\infty}(U_p)$ which vanish to order $k$ on $S$, all indexed by $1\leq j\leq N$, such that
$\sigma=\sum_{1\leq j\leq N} f_j D_j h_j$ on $U_p$. We say that $\sigma$ vanishes to infinite order along $S$ if it vanishes to order $k$ for every $k\geq 0$.
\end{definition}
By \cite[Lemma 2.8]{miller2004distributions}, a distribution $\sigma$ on $M$ that vanishes to infinite order along $S$ is uniquely determined by its restriction to $M\smallsetminus S$. Thus we have the following terminology:
\begin{definition}[Definition 2.16 of \cite{miller2004distributions}]
  A distribution $\tau$ defined on the complement $M\smallsetminus S$ of a closed submanifold $S\subset M$ has a canonical extension across $S$ if there exists a---necessarily unique---distribution $\sigma \in C^{-\infty}(M)$ that vanishes to infinite order along $S$ and agrees with $\tau$ on $M\smallsetminus S$.  
\end{definition}
The following lemma provides a prototypical example of a distribution vanishing to infinite order. The proof uses an application of the chain rule. This technique is also used in \cite{kim2023infinitely}, \cite{miller2004distributions}, and \cite{miller2008rankin}.
\begin{lem}\label{e1/x}
    The distribution $e(1/x)\in C^{-\infty}(\R\smallsetminus \{0\})$ has a canonical extension across $0$.
\end{lem}
\begin{proof}
    For any $k\geq 0$, we have
    $e(1/x)=(-\frac{x^2}{2\pi i}\frac{d}{dx})^k e(1/x)$.
    Expanding the differential operator gives
\begin{equation}\label{tauchain2}
    e(1/x)=\sum_{j=1}^{k} q_j x^{j+k}\frac{d^j}{dx^j}e(1/x),
\end{equation} 
where $q_j$ are constants. The right hand side is an extension of $e(1/x)$ which vanishes to order $k$ at $0$. Since the choice of $k\geq 0$ is arbitrary, the lemma follows.
\end{proof}
Via a multivariable analogue of the technique used in the above lemma, our main theorem implies that a Whittaker distribution $\tau$ vanishes to infinite order along the complement $G\smallsetminus(NB_{-})$. In the remainder of this section, we apply the technique to Jacquet integrals to describe their analytic continuation.

\subsection{Jacquet integrals for $\GL(2)$}
We first use $G=\GL(2,\R)$ as a basic example, in which the underlying idea is easy to understand. The following example can also be found in \cite{casselman2000bruhat}. Let $\mu\in\C$ and let $\chi_{\mu}$ be the character on $B_{-}$ defined by $\chi_{\mu}\left(\begin{smallarray}{cc}
 a & 0 \\
 *& d \\
\end{smallarray}
\right)=\abs{\frac{a}{d}}^\mu$. Let $V_{\mu}$ be the principal series representation induced from $\chi_{\mu}$ and let $V^\infty_\mu$ be its subspace consisting of smooth functions. For simplicity, consider a spherical function $f\in V_\mu^{\infty}$ satisfying $f(kb_{-})=\chi_\mu(b_{-})$ for all $k\in K=SO(2,\R)$. A matrix in $N$ decomposes as
\[
 \left(
\begin{smallarray}{cc}
 1 & x \\
 0 & 1 \\
\end{smallarray}
\right)=   \Bigl(
\begin{smallarray}{cc}
 \sqrt{x^2+1}\i & x \sqrt{x^2+1}\i \\
 -x \sqrt{x^2+1}\i & \sqrt{x^2+1}\i \\
\end{smallarray}
\Bigr)\Bigl(
\begin{smallarray}{cc}
 \sqrt{x^2+1}\i & 0 \\
 x\sqrt{x^2+1}\i & \sqrt{x^2+1} \\
\end{smallarray}
\Bigr), 
\]
so the Jacquet integral of $f$ equals
\begin{equation}\label{GL2Jacquet}
   \langle \tau,f\rangle _{\mu}=\int_{-\infty}^\infty e^{2\pi i x}f\Bigl(\begin{smallarray}{cc}
 1 & x \\
 0 & 1 \\
\end{smallarray}
\Bigr) dx=\inR e^{2\pi i x}\Bigl(\frac{1}{x^2+1}\Bigr)^\mu dx.
\end{equation}
This integral converges for $\re \mu >\frac{1}{2}$. Our goal is to meromorphically continue this integral to the entire complex plane. \par 
Recall that $G$ can be written as the overlapping union $G=NB_{-}\cup wNB_{-}$ with $w=\Bigl(
\begin{smallarray}{cc}
 0 & 1 \\
 1 & 0 \\
\end{smallarray}
\Bigr)$. We shall write $f$ as a sum $f=f_1+f_w$, where $f_1$ has compact support on $NB_{-}$ modulo $B_{-}$ and $f_w$ has compact support on $wNB_{-}$ modulo $B_{-}$. Let $\varphi$ be a bump function satisfying $\phi(x)\equiv 1$ near $0$. Let $f_1$ and $f_w$ restricted on $N$ be
\begin{equation}\label{f_1andf_w}
    f_1\left(\begin{smallarray}{cc}
 1 & x \\
 0 & 1 \\
\end{smallarray}
\right)=\Bigl(\frac{1}{x^2+1}\Bigr)^\mu \varphi(x) \quad \text{and}\quad f_w\left(\begin{smallarray}{cc}
 1 & x \\
 0 & 1 \\
\end{smallarray}
\right)=\Bigl(\frac{1}{x^2+1}\Bigr)^\mu \bigl(1-\varphi(x)\bigr).
\end{equation}
It is clear that $f_1$ has compact support on $NB_{-}$ modulo $B_{-}$. Write $h_1(x)=f_1\left(
\begin{smallarray}{cc}
 1 & x \\
 0 & 1 \\
\end{smallarray}
\right)$. Also, since $w\Bigl(
\begin{smallarray}{cc}
 1 & x \\
 0 & 1 \\
\end{smallarray}
\Bigr)=\Bigl(
\begin{smallarray}{cc}
 1 & 1/x \\
 0 & 1 \\
\end{smallarray}
\Bigr)\Bigl(
\begin{smallarray}{cc}
 -1/x & 0 \\
 1 & x \\
\end{smallarray}
\Bigr)$, we have
$f_w\Bigl(w\left(\begin{smallarray}{cc}
 1 & x \\
 0 & 1 \\
\end{smallarray}
\right)\Bigr)=\Bigl(\frac{1}{x^2+1}\Bigr)^\mu \bigl(1-\varphi(1/x)\bigr)$.
Observe that $1-\varphi(1/x)$ has compact support near $0$. It follows that $f_w$ has compact support on $wNB_{-}$ modulo $B_{-}$. Write $h_w(x)=f_w\Bigl(w\left(\begin{smallarray}{cc}
 1 & x \\
 0 & 1 \\
\end{smallarray}
\right)\Bigr)$. We have
\begin{equation}\label{GL2Jacquet2}
\begin{aligned}
    \langle \tau, f\rangle _\mu &=\inR e^{2\pi i x}f\Bigl(\begin{smallarray}{cc}
 1 & x \\
 0 & 1 \\
\end{smallarray}
\Bigr)\varphi(x)\dx+\inR e^{2\pi i x}f\Bigl(\begin{smallarray}{cc}
 1 & x \\
 0 & 1 \\
\end{smallarray}
\Bigr)\bigl(1-\varphi(x)\bigr)\dx \\ 
&=\inR e^{2\pi i x}h_1(x)\dx+\inR e^{\frac{2\pi i} {x}}f\Bigl(\begin{smallarray}{cc}
 1 & 1/x \\
 0 & 1 \\
\end{smallarray}
\Bigr)\bigl(1-\varphi(1/x)\bigr)\,\frac{dx}{x^2} \\
&=\inR e^{2\pi i x}h_1(x)\dx+\inR e^{\frac{2\pi i}{x}}h_w(x)\abs{x}^{2\mu-2}dx.
\end{aligned}
\end{equation}
Since both $h_1(x)$ and $h_w(x)$ have compact support on $\R$, the first integral converges for all $\mu\in\C$ and the second integral converges for $\re \mu> \frac{1}{2}$. However, since $e(1/x)$ vanishes to infinite order at $0$, we can meromorphically continue the second integral to all $\mu\in\C$ by writing $e(1/x)$ as in \eqref{tauchain2} and applying integration by parts.

\subsection{Jacquet integrals for $\GL(r)$ with $r>2$}\label{sec:jacquetr} We proceed to prove Theorem~\ref{thm:existence} using the idea presented in the previous section. Let $f\in V_{-\l,\d}^\infty$. By Lemma~\ref{wNB}, we can write $G/B_{-}$ as an overlapping union $G/B_{-}=\bigcup_{w\in \Omega}wNB_{-}/B_{-}$. Therefore, using a partition of unity, we can write $f$ as the sum $\sum_{w\in\Omega} f_w$, where $f_w$ has compact support on $wNB_-$ modulo $B_-$. Write
\begin{equation} \label{taupairingsum}
    \langle \tau,f\rangle _{\l,\d}=\sum\limits_{w\in\Omega} \langle \tau, f_w\rangle _{\l,\d}=\sum\limits_{w\in\Omega}\int_N \tau(wn)f_w(wn)\dn.
\end{equation}
By \eqref{nws}, we can write $n\in N$ as $n=xu=w^\i x'w u$, where $x\in N_{w\i}$, $u\in N^{w\i}$, and $x'=w x w\i\in N_w$. We have $\tau(wn)=\tau(x' w u)=\psi(x') \tau(wu)$, thus
\begin{equation}\label{tauwnint}
    \int_N \tau(wn)f_w(wn)\dn
     =\int_{N_{w\i}} \psi(x')\Bigl(\int_{N^{w\i}} \tau(w u) f_w(w x u) \, du\Bigr) \, dx,
\end{equation}
where $du=\prod_{\a\in \Inv(\pi\i)} du_\a$ and $dx=\prod_{\substack{1\leq i<j\leq r \\ (i,j)\notin \Inv(\pi\i)}} dn_{i,j}$. We can regard the inner integral as an integral over the region ${\mathcal{C}}=[-M,M]^d$ and the outer over $\mathcal{D}=[-M,M]^{\frac{1}{2}r(r-1)-d}$ for some $M>0$, where $d=\abs{V_w}$. Next, we apply the change of coordinates $u_\a=R(n_\a)$ given by Theorem~\ref{biratlthm} and write the resulting matrix as $u=R(n)$. It is a direct consequence of parts (ii) and (iii) of Theorem~\ref{biratlthm} that after the change of coordinates $\tau$ takes the simple form
\[\tau(wu)=\pm e\Bigl(\sum_{\a\in \Inv(\pi\i)}\frac{1}{n_\a}\Bigr)\prod_{\a\in \Inv(\pi\i)} \abs{n_{\a}}^{s_\a}\sgn (n_\a)^{\eta_\a},\]
where each $s_\a\in \C$ and each $\eta_\a\in \Z/2\Z$ is a linear expression of the parameters $\lambda_j\in \C$ and a linear combination of $\delta_j\in \Z/2\Z$, respectively, for $1\leq j\leq r$, whose expression depends on $w$. The sign $\pm$ depends only on $w$ and $\delta$. Together with part (iv) of Theorem~\ref{biratlthm}, this implies that the inner integral of \eqref{tauwnint} becomes
\begin{equation}\label{fwJacquet2}
\begin{aligned}
  \pm \int_{R^{-1}(\mathcal{C})} f_w\bigl(w x R(n)\bigr) \Bigl(\prod_{\a\in \Inv(\pi\i)} e\Bigl(\frac{1}{n_\a}\Bigr) \abs{n_\a}^{t_\a} \sgn(n_\a)^{\eta_\a}\Bigr) \Bigl(\prod_{\a\in \Inv(\pi\i)} dn_\a\Bigr),
     \end{aligned}
\end{equation} 
where each $t_\alpha$ is a linear expression in $\lambda_j$, $1\leq j\leq r$. By part (i) of Theorem~\ref{biratlthm}, the function $f_w\bigl((w n_\b R(n)\bigr)$ is smooth when all $n_\a\in V_w$ are nonzero. Also, we can cancel out the possible poles of $\abs{n_\a}^{t_\a}$ at $n_\a=0$ by writing the $e(1/n_\a)$ as in \eqref{tauchain2} with sufficiently large $k$, and applying integration by parts. \par 
However, the situation is subtle when there is more than one free variable in $V_w$. When we integrate by parts for $n_\a$ to increase the power of $n_\a$ in the expression \eqref{tauchain2}, we take the partial derivative $\partial /\partial n_\alpha$ to the rest of the integrand which involves other variables $n_{\b}\in V_w$. Part (v) of Theorem~\ref{biratlthm} gives us the correct order of integration. Write $V_w=\{n_{\a_1},n_{\a_2},\ldots,n_{\a_d}\}$, where $n_{\a_1}\succ n_{\a_2}\succ \cdots \succ n_{\a_d}$. Taking the derivative in terms of $n_{\a_j}$ changes the exponents of some other $\vert n_{\a_i} \vert$. However, by part (v) of Theorem~\ref{biratlthm}, such changes only occur to those $n_{\a_i}$ such that $n_{\a_i}\prec n_{\a_j}$. Therefore, by performing the integration by parts in the sequence $n_{\a_1}, n_{\a_2}, \ldots , n_{\a_d}$, we can compensate the exponents of $\vert n_{\a_i} \vert$ which have been decreased by the earlier integration by parts via \eqref{tauchain2}, and eventually make the exponents of all $\vert n_\a \vert$ be positive. This process constructs an extension of $\tau(wu)$ vanishing to arbitrarily high order at any $n_\a=0$. After the process, the resulting integral is of the form 
\begin{equation}\label{fwibp}
\int_{\mathcal{D}} \psi(x')\int_{R^{-1}(\mathcal{C})} \widetilde{f}_w\bigl(w x R(n)\bigr) \Bigl(\prod_{\a\in \Inv(\pi\i)} e\Bigl(\frac{1}{n_\a}\Bigr) \abs{n_\a}^{t'_\a} \sgn(n_\a)^{\eta'_\a}\Bigr) \prod_{\a\in \Inv(\pi\i)} dn_\a\, dx,
\end{equation}
where $\widetilde{f}_w\bigl(w n\bigr)$ is a smooth function of compact support on $wNB_{-}/B_{-}$ and $\re(t'_\a)>0$ for all $\alpha\in \Inv(\pi\i)$. The exponents $t'_\alpha$ remain linear expressions in $\lambda_j$ with $1\leq j\leq r$. The inner integral is bounded by a constant multiple of
\begin{equation} \label{innerint} \int_{R^{-1}(\mathcal{C})} \Bigl(\prod_{\a\in \Inv(\pi\i)}\abs{n_\a}^{s_\a}\Bigr) \prod_{\a\in \Inv(\pi\i)} dn_\a \end{equation}
with $s_\alpha>0$, uniformly in $x\in \mathcal{D}$. For the convergence of this integral, we use the following lemma. 
\begin{lem}\label{lem:intdom}
There exists an ordering $\sqsupset$ on $V_w$,
\[n_{\beta_1}\sqsupset n_{\beta_2}\sqsupset \cdots \sqsupset n_{\beta_d}, \quad n_{\beta_j}\in V_w,\]
and rational functions $h_1,\ldots,h_d$ with the following properties:
\begin{itemize}
    \item[(i)] $h_d\equiv M$,
    \item[(ii)] for $1\leq j<d$, $h_j$ can be written in the form $C_j/D_j$, where $C_j$ is a $\Z$-linear combination of monomials of the variables $n_{\beta_{j+1}},\ldots,n_{\beta_d}$, $D_j$ is a monomial of $n_{\beta_{j+1}},\ldots,n_{\beta_d}$, and the degree of $n_{\beta_i}$ in each monomial of $C_j$ and $D_j$ does not exceed $1$,
    \item[(iii)] $R\i(\mathcal C)\cap (\R^{\times})^d$ is contained in the region
    \[\widetilde{\mathcal{C}}=\left\{(n_{\beta_1},\ldots,n_{\beta_d})\in (\R^{\times})^d \,\Big\vert\, n_{\beta_j}\in I_j\text{ for all } 1\leq j\leq d\right\},\]
    where
    \[I_j=\left[-h_j\left(\vert n_{\beta_{j+1}}\vert ,\ldots,\vert n_{\beta_d}\vert \right),h_j\left(\vert n_{ \beta_{j+1}}\vert ,\ldots,\vert n_{\beta_d}\vert \right)\right].\]
\end{itemize}
\end{lem}
The ordering $\sqsupset$ is used later in Section~\ref{sec:part4} to compute the determinant of the Jacobian matrix of $R$ and prove part $(iv)$ of Theorem~\ref{biratlthm}. We therefore postpone the proof of this lemma to Section~\ref{sec:intdom} at the very end of this paper, and deduce here the convergence of the integral \eqref{innerint} from Lemma~\ref{lem:intdom}. The lemma implies that the integral \eqref{innerint} is bounded by the integral
\begin{equation}\label{innerint2}
\int_{\widetilde{\mathcal{C}}} \Bigl(\prod_{1\leq j\leq d}\abs{n_{\beta_j}}^{s_{\beta_j}}\Bigr) \prod_{1\leq j\leq d} dn_{\beta_j}\end{equation}
and that the above integral can be performed in the order of $n_{\beta_1},\ldots,n_{\beta_d}$. It suffices to consider the case where each $s_\alpha$ is a positive integer, because $\abs{n_{\beta_j}}^{s_{\beta_j}}\ll 1$ near $0$ and $\abs{n_{\beta_j}}^{s_{\beta_j}} \ll \abs{n_{\beta_j}}^{\lceil s_{\beta_j} \rceil}$ away from $0$.

We perform the integral in the order $n_{\b_1},\ldots,n_{\b_d}$. The integral with respect to $n_{\b_1}$ converges as long as $s_{\b_1}>0$. After the integration, the integral becomes a linear combination of integrals of the form
\[\int_{I_d}\int_{I_{d-1}}\cdots \int_{I_{2}} \Bigl(\prod_{2\leq l\leq d}\abs{n_{\beta_l}}^{s^{(1)}_{\beta_l}}\Bigr) \prod_{2\leq l\leq d} dn_{\beta_l}, \]
where $s_{\b_l}-s_{\b_1}\leq s^{(1)}_{\b_l}\leq s_{\b_l}+s_{\b_1}$. Take one such integral. The integral with respect to $n_{\b_2}$ converges as long as $s_{\b_2}\geq s_{\b_1}$. After the integral over $n_{\b_2}$, the integral becomes a linear combination of the integrals of the form
\[\int_{I_d}\int_{I_{d-1}}\cdots \int_{I_{3}} \Bigl(\prod_{3\leq l\leq d}\abs{n_{\beta_l}}^{s^{(2)}_{\beta_l}}\Bigr) \prod_{3\leq l\leq d} dn_{\beta_l}, \]
with 
\[s^{(1)}_{\b_l}-s^{(1)}_{\b_2}\leq s^{(2)}_{\b_l}\leq s^{(1)}_{\b_l}+s^{(1)}_{\b_2}, \quad 3\leq l\leq d.\]
It follows that the integral with respect to $n_{\b_3}$ converges if \[s_{\b_3}^{(2)}\geq s_{\b_3}^{(1)}-s_{\b_2}^{(1)}\geq (s_{\b_3}-s_{\b_1})-(s_{\b_2}+s_{\b_1})\geq 0,\]
that is, as long as $s_{\b_3}\geq s_{\b_2}+2s_{\b_1}$. Proceeding this way, we see that the integral \eqref{innerint2} converges as long as we have \[s_{\beta_j}\geq \sum_{1\leq i<j} (j-i) s_{\beta_i}\]
for each $1\leq j\leq d$. Therefore, by adjusting the powers $t_\alpha$ of the variables $n_\alpha\in V_w$ in \eqref{fwJacquet2} by integration by parts to $t'_\alpha$ so that $s_\alpha=\re t'_\alpha$ satisfy the above bound, we arrive at the absolutely convergent expression \eqref{fwibp}. Since each $t_\alpha$ depends linearly on $\lambda$, for any compact $K\subset \C^r$ we may choose the number of integration by parts for each variable uniformly in $\lambda\in K$. Therefore, the integral in \eqref{fwibp} defines a holomorphic function of $\lambda$ in $\C^r$. Moreover, for $\lambda\in U$, the change of variables and the integration by parts can be reversed, hence \eqref{tauwnint} and \eqref{fwibp} agree on $U$. Therefore, by summing the expression in \eqref{fwibp} over all $w\in \Omega$, we obtain an analytic continuation of the Jacquet integral. The $N$-equivariance of the resulting expression follows immediately from the $N$-equivariance of the original expression for $\lambda\in U$. This completes the proof of Theorem~\ref{thm:existence}.
\subsection{Application: Bessel functions on $\GL(r)$} \label{sec:Bessel}
The classical Bessel functions are defined as the solutions $y=y(x)$ to the Bessel differential equation
\[x^2\frac{d^2y}{dx^2}+x\frac{dy}{dx}+(x^2-\alpha^2)y=0.\]
Bessel functions first appeared in the work of D. Bernoulli in 1732, and a systematic theory of their properties were first developed by Bessel \cite{dutka1995early}, \cite{watson1922treatise}. The Bessel functions arise naturally in the theory of automorphic forms and are closely related to Whittaker functions. The $\GL(2)$ Whittaker function, which appears as the Fourier coefficients of Maass cusp forms, can be expressed in terms of $K$-Bessel functions. Indeed, the $\GL(2)$ Jacquet integral \eqref{GL2Jacquet} above satisfies
\[\inR e^{2\pi i x}\Bigl(\frac{1}{x^2+1}\Bigr)^\mu dx=\frac{2\pi^\mu}{\Gamma(\mu)}K_{\mu-\frac{1}{2}}(2\pi).\]
Bessel functions also appear in the trace formulas of Petersson \cite{petersson1932entwicklungskoeffizienten} and Kuznetsov \cite{kuznetsov1981petersson}, as well as in the Voronoi summation formula \cite{voronoi1904voronoi}. These summation formulas are among the most powerful tools in the analytic theory of automorphic forms, with applications including subconvexity bounds for $L$-functions and equidistribution problems. \par 
There have been generalizations of the Voronoi summation formula \cite{goldfeld2008voronoi}, \cite{miller2011voronoi}, \cite{miller2019balanced} and the Kuznetsov trace formula \cite[Section 11.6]{goldfeld2006automorphic} (worked out by Xiaoqing Li), \cite{blomer2013applications}, \cite{buttcane2016spectral} to the groups $\GL(r)$ for $r>2$, and these developments naturally lead to a generalization of classical Bessel functions in the $\GL(r)$ setting. In \cite{buttcane2024bessel}, Buttcane outlined the theory of Bessel functions on the group $\GL(r,\R)$ and formulated a series of conjectures towards the construction of Kuznetsov-type formulas. In \cite{buttcane2025bessel}, he verified his conjectures for the group $\GL(4,\R)$. Among his conjectures, the birational map $R$ from Theorem~\ref{biratlthm} best applies to the Direct Integral Representation Conjecture \cite[Section 5.5.1]{buttcane2024bessel}. Using the notation from our paper, the conjecture states that the integral 
\begin{equation}\label{ButtcaneIw}
\mathcal{I}_w(y,\mu,\delta)=\int_{N^{w\i}}\tau(ywu)\psi^{-1}(u)\du\end{equation}
defined for invertible diagonal matrices $y$, where $\tau\in V_{\mu,\delta}^{-\infty}$, gives the Bessel function. \par
This integral fails to be absolutely convergent for any value of $\mu$. In \cite{qi2020fundamental}, Qi studied the Bessel functions in the Voronoi formula for $\GL(r)$, which corresponds to the integral $\mathcal{I}_{w_{r-1,1}}(y,\mu,\delta)$, where $w_{r-1,1}=\left(\begin{smallmatrix} & I_{r-1} \\ 1 & \end{smallmatrix}\right)$. A simple change of variable reduces $\mathcal{I}_{w_{r-1,1}}(y,\mu,\delta)$ to an integral of the form 
\begin{equation}\label{QiVor}\int_{\mathbb{R}_+^{r-1}} \Bigl(\prod_{l=1}^{r-1} x_l^{\nu_l - 1}\Bigr) e^{iy(\zeta_{r} x_1\cdots x_{r-1} + \sum_{l=1}^{r-1} \zeta_l x_l^{-1})} \, dx,\end{equation}
where $y>0$, $\boldsymbol{\nu}=(\nu_1,\ldots,\nu_{r-1})\in \C^{r-1}$, and $\boldsymbol{\zeta}=(\zeta_1,\ldots,\zeta_r)\in \{+1,-1\}^r$. Qi gives a rigorous interpretation of the above integral by writing it as a linear combination of absolutely convergent integrals, using partition of unity and integration by parts. Also, Qi used stationery phase method to study the asymptotics of the integral. \par
The change of variables that transforms \eqref{ButtcaneIw} into \eqref{QiVor} is exactly the birational map $R_{w_{r-1,1}}$ that we provide in this paper, and we provide such a birational map for every other Weyl group element of $\GL(r)$. In particular, applying the birational map $R_w$ reduces $\mathcal{I}_w$ for any $w\in \Omega$ to an oscillatory integral of the form
\begin{equation}\label{Iwnew}\int_{N^{w\i}} \Bigl(\prod_{l=1}^{d} \abs{n_l}^{\nu_l - 1}\Bigr) e^{iy(\zeta_{d+1} f(n_1,\ldots,n_d) + \sum_{l=1}^{d} \zeta_l n_l^{-1})} \, dn,\end{equation}
where $V_w=\{n_1,\ldots,n_d\}$, and $f$ is a rational function that has a simple description in terms of the birational map $R_w$. This simplification allows us to extend Qi's methods to the Bessel integrals \eqref{ButtcaneIw} associated to Weyl group elements other than $w_{r-1,1}$. There is another advantage of the birational map $R_w$ towards the direction of Mellin-Barnes integral: Buttcane provides a Mellin-Barnes integral formula for $\mathcal{I}_{w_{r-1,1}}$ as a conjecture \cite[Conjecture 10]{buttcane2024bessel} and obtains the Mellin-Barnes integral of $\mathcal{I}_w$ for Weyl group elements of $\GL(4)$ in \cite{buttcane2025bessel} using integral formulas derived from \cite[3.871.1-4, 8.403.1, 17.43.16, 17.43.18]{gradshteyn2015table}. The simplied form \eqref{Iwnew} makes the integral formulas readily applicable for certain types of $w$, allowing its Mellin-Barnes integral to be computed. We leave the further development of this direction for future work.

\section{Formula for the birational map}\label{sec:Rformula}
In this section, we give an explicit formula for the map $R$. Most of this section is repeated from Section~2 of \cite{kim2024pt1}, but we include it here for readability and context. See \cite{kim2024pt1} for examples of the definitions in this section.
\subsection{Notation}
For a Weyl group element $w\in \Omega$ in $G=\GL(r)$, we let $wn$ be the general matrix of $wN^{w\i}$, where $n\in N^{w\i}$ is the unipotent upper triangular matrix with free variables $n_{a,b} \in V_w$, as defined in (\ref{nws}) and (\ref{coords}). Below is an example of $wn$ for a Weyl group element $w$. 
\begin{equation}\label{ex:w3}
    w=\left(
\begin{smallarray}{ccccc}
 0 & 1_2 & 0 & 0 & 0 \\
 0 & 0 & 0 & 0 & 1_5 \\
 0 & 0 & 0 & 1_4 & 0 \\
 1_1 & 0 & 0 &0 & 0 \\
 0 & 0 & 1_3 &0& 0 \\
\end{smallarray}
\right), \quad wn=\left(
\begin{smallarray}{ccccc}
 0 & 1_2 & 0 & 0 & 0 \\
 0 & 0 & 0 & 0 & 1_5 \\
 0 & 0 & 0 & 1_4 & n_{4,5} \\
 1_1 & n_{1,2} & 0 & n_{1,4} & n_{1,5} \\
 0 & 0 & 1_3 & n_{3,4} & n_{3,5} \\
\end{smallarray}
\right).\end{equation}
To distinguish the entries of $1$, we assign indices to them and write them as $1_1, 1_2, \ldots , 1_r$ from the left to the right. Note that the free variables appear in the entries that are below the entry of $1$ in its column, and to the right of the entry of $1$ in its row. Recall from \eqref{wnentries} that $n_{a,b}$ is in the $\pi\i(a)$-th row and the $b$-th column of $wn$. Since $1_j=n_{j,j}$, $1_j$ is in the $\pi\i(j)$-th row and the $j$-th column. By \eqref{ordering}, we have $n_{1,2}\succ n_{1,4}\succ n_{3,4}\succ n_{4,5}\succ n_{1,5}\succ n_{3,5}$.
\par
We define rectilinear paths in the matrix $wn$ as follows: we say that two nonzero entries of $wn$ are \textit{neighbors} if they are in the same row or the same column, and if there is no other nonzero entry between them. This way we can regard $wn$ as a directed graph, as illustrated above. For two nonzero entries $n_\a$ and $n_\b$ in $wn$, possibly $1$, a \textit{path} from $n_{\alpha}$ to $n_{\beta}$ is a sequence of nonzero neighboring entries from $n_\alpha$ to $n_{\beta}$ which moves only in upward and leftward steps as it passes through the nonzero entries of $wn$. We define the length of a path $p$ as the number of entries in $p$ minus $1$. \par 
We say that two paths $p_1$ and $p_2$ in $wn$ are \textit{disjoint} if $p_1$ and $p_2$ do not share any entries of $wn$, and that $p_1$ and $p_2$ \textit{intersect} if they have one or more common elements. If $p_1$ and $p_2$ both contains $n_\a\in wn$, we say that $p_1$ and $p_2$ \textit{intersect} at $n_\a$. Extending the notion of disjointedness to sets of paths, we say that $\mathbf p=\{p_1, p_2, \ldots, p_m\}$ is a set of disjoint paths if $p_i$ and $p_j$ are disjoint whenever $i\neq j$. For two sets of disjoint paths $\mathbf p$ and $\mathbf q$, we say that $\mathbf p$ and $\mathbf q$ are disjoint if any two paths $p\in \mathbf p$ and $q\in \mathbf q$ are disjoint.
\begin{definition} Let $A$ and $B$ be sets of nonzero entries of $wn$ such that $\abs{A}=\abs{B}=m\geq 1$. We let ${\mathcal P}\left(A\to B\right)$ denote the collection of all possible sets of $m$ disjoint paths connecting elements of $A$ to elements of $B$.
\end{definition}
This means that each of the $m$ elements of $A$ is connected by a path to a unique element of $B$, in a way that no two paths have an entry of $wn$ in common. If there is an entry of $wn$ that is both in $A$ and $B$, then the entry itself forms a path of length $0$.
\begin{definition}\label{def:upath}
For any set of disjoint paths $\mathbf{p}$, we let $u(\mathbf{p})$ denote the product of all free variables that are traversed by any of the $m$ paths.
\end{definition} 
\begin{definition}\label{pathsumdef}
We let
$P(A\to B)=\sum\limits_{\mathbf{p}\in{\mathcal P}\left(A\to B\right)}u(\mathbf{p})$.
\end{definition}
\begin{definition}\label{def:gamma1}
For each $1_j$ in $wn$, we let $\gamma(1_j)$ denote the rightmost nonzero entry of the row containing $1_j$. For a set $I$ consisting of the entries $1_j$, we let $\gamma(I)=\{\gamma(1_i) \mid 1_i\in I\}$.
\end{definition}
\begin{definition}\label{rhodef}
For each $1_j$ in $wn$, we let $\rho(1_j)$ denote the product of nonzero entries in the row which contains $1_j$.
\end{definition}
\subsection{Formula for $R_w$}\label{subsec:Rformula}
For each $n_{a,b}\in V_w$, we assign ``origins" and ``destinations" corresponding to the variable. Recall that $1_{b}$ is the entry of $1$ above $n_{a,b}$. We call $1_{b}$ the ``top destination" of $n_{a,b}$.
\begin{definition}\label{mwna}
The matrix $M_w(n_{a,b})$ denotes the submatrix of $wn$ whose rows consist of the $\pi\i(b)$-th row through the $\pi\i(a)$-th row, and whose columns consist of the $b$-th column through the $r$-th column.
\end{definition}
Observe that the top left corner entry of $M_w(n_{a,b})$ is $1_b$ and the bottom left corner is $n_{a,b}$.
\begin{definition}\label{Ddef}
The set of ``destinations" of $n_{a,b}$, which we denote by $D(n_{a,b})$, is the set of entries of $1$ in the submatrix $M_w(n_{a,b})$.
\end{definition} 
\begin{definition}\label{Odef}
The set of ``origins" of $n_{a,b}$, which we denote by $O(n_{a,b})$, is the set consisting of the rightmost nonzero entry of the bottom row of $M_w(n_{a,b})$, and the rightmost nonzero entries of the rows of $M_w(n_{a,b})$ that contain a destination of $n_{a,b}$ other than $1_b$.
\end{definition}
We call the origin in the bottom row of $M_w(n_{a,b})$ the ``bottom origin". For each row that contains a destination of $n_{a,b}$ other than the top destination, there is also an origin of $n_{a,b}$ which is not the bottom origin. Thus $D(n_{a,b})$ and $O(n_{a,b})$ are nonempty and have the same size, hence the collection ${\mathcal P}\bigl(O(n_\alpha)\to D(n_\alpha)\bigr)$ and the polynomial $P\bigl(O(n_\alpha)\to D(n_\alpha)\bigr)$ are well-defined. 
\begin{definition}\label{def:calPn}
We let ${\mathcal P}(n_\alpha)={\mathcal P}\bigl(O(n_\alpha)\to D(n_\alpha)\bigr)$
and $P(n_\alpha)=P\bigl(O(n_\alpha)\to D(n_\alpha)\bigr)$.
\end{definition}
The formula for the birational map $R$ is given as follows.
\begin{definition}\label{def:Rformula}
    For $n_\a\in V_w$, we let
    \begin{equation}\label{defR}
R(n_\a)=(-1)^{t_\a}\frac{P(n_\a)}{\prod\limits_{1_\mu\in D(n_\alpha)}\rho(1_\mu)},
\quad \text{where}\quad t_\a=\abs{O(n_\alpha)}-1=\abs{D(n_\alpha)}-1.\end{equation}
\end{definition}
We shall use the notations $R(n_\a)$ and $u_\a$ interchangeably. We let $wu$ denote the matrix obtained from $wn$ by replacing $n_\a\in V_w$ with $u_\a$, and denote 
$u_{a,b}=(wu)_{\pi\i(a),b}$. Thus $u_{a,b}=n_{a,b}$ if $n_{a,b}$ is either $0$ or $1$, and $u_{a,b}=R(n_{a,b})$ if $n_{a,b}\in V_w$. See Section 1.2 and (2.13) in \cite{kim2024pt1} for examples of the birational map $R$.

\section{Determinant of submatrices}\label{sec:detsub}
Throughout the section, we fix a Weyl group element $w\in \Omega$ of $\GL(r)$. In this section, we state the propositions from which we deduce parts (ii) and (iii) of Theorem~\ref{biratlthm}.
\subsection{Further notation for sets, paths, and matrices}
\begin{definition}\label{O1}
We let $O_1(n_\alpha)\subseteq O(n_\a)$ denote the subset consisting of the origins which are (i) in the rightmost column of $wn$, (ii) above the bottom origin, and (iii) right below a free variable which is not an origin of $n_\a$. That is,
\[O_1(n_{a,b})=\{n_{\pi(t),r}\in O(n_{a,b})\mid t<\pi\i(a),\  n_{\pi(t-1),r}\in V_w\smallsetminus O(n_{a,b})\}.\]
\end{definition}
Observe that every element of $O_1(n_{a,b})$ is a free variable because they are right below a free variable.
\begin{definition}\label{D1}
We let $D_1(n_\a)\subseteq D(n_\a)$ denote the subset consisting of the entries of $1$ in the rows of $wn$ containing an element of $O_1(n_\a)$.
\end{definition}
\begin{definition}\label{ODup}
For a positive integer $k\leq r$, we let $O_{\uparrow,k}(n_\a)$ and $D_{\uparrow,k}(n_\a)$ denote the set of origins and destinations of $n_\a$ which are in or above the $k$-th row of $wn$, respectively. Similarly, we let $O_{\downarrow,k}(n_\a)$ and $D_{\downarrow,k}(n_\a)$ denote the set of origins and destinations of $n_\a$ which are in or below the $k$-th row of $wn$, respectively.
\end{definition}
\begin{definition}\label{calpleft}
We let ${\mathcal P}_L\left(A\to B\right)\subseteq {\mathcal P}\left(A\to B\right)$ denote the subcollection consisting of the sets in which every path with length at least $1$ whose origin is in the rightmost column of $wn$ starts by going left. We let
$P_L(A\to B)=\sum_{\mathbf{p}\in{\mathcal P}_L\left(A\to B\right)}u(\mathbf{p})$.
\end{definition}
\begin{definition}\label{def:calpleftn}
We let ${\mathcal P}_L(n_\a)={\mathcal P}_L\bigl(O(n_\a)\to D(n_\a)\bigr)$ and $P_L(n_\a)=P_L\bigl(O(n_\a)\to D(n_\a)\bigr)$.
\end{definition}
Since the entries in the rightmost column of $wn$ are zero if they are above $1_r$ and free variables if they are below $1_r$, the elements of $O(n_\a)$ are in the rightmost column if and only if they are below the $\pi\i(r)$-th row. If all of the origins of $n_\a$ are above the $\pi\i(r)$-th row, then ${\mathcal P}(n_\a)={\mathcal P}_L(n_\a)$. Write ${\mathcal P}'(n_\a)={\mathcal P}(n_\a)\smallsetminus {\mathcal P}_L(n_\a)$.
\begin{definition}\label{P1def}
Let ${\mathcal P}_1(n_\a)$ denote the subcollection of ${\mathcal P}'(n_\a)$ consisting of sets in which the path from the bottom origin starts by going left. Let $P_1(n_\a)=\sum_{\mathbf{p}\in{\mathcal P}_1(n_\a)}u(\mathbf{p})$.
\end{definition}
\begin{definition}\label{P2def}
Let ${\mathcal P}_2(n_\a)$ denote the subcollection of ${\mathcal P}'(n_\a)$ consisting of sets in which the path from the bottom origin starts by going up. Let $P_2(n_\a)=\sum_{\mathbf{p}\in{\mathcal P}_2(n_\a)}u(\mathbf{p})$.
\end{definition}
\noindent It is clear that ${\mathcal P}'(n_\a)={\mathcal P}_1(n_\a)\sqcup{\mathcal P}_2(n_\a)$, thus $P(n_\a)=P_L(n_\a)+P_1(n_\a)+P_2(n_\a)$.
\begin{definition}\label{def:RLR1R2}
For a subscript $\bullet \in \{L,1,2\}$, we let
\[
  R_{\bullet}(n_\a)=(-1)^{t_\a} \frac{P_{\bullet}(n_\a)}{\prod\limits_{1_\mu\in D(n_\alpha)}\rho(1_\mu)}, \quad \text{where}\quad t_\a=\abs{D(n_\alpha)}-1.
\]
\end{definition}
Using this notation, we can write $R(n_\a)$ as the sum
\begin{equation}\label{Rpartition}
R(n_\a)=R_L(n_\a)+R_1(n_\a)+R_2(n_\a). 
\end{equation}
For parts (ii) and (iii) of Theorem~\ref{biratlthm}, we use the following lemma.
\begin{lem}[Corollary~4.3 of \cite{miller2012archimedean}]\label{UDLlem}
If $g =n h n_-$, with $n$, $h$, $n_-$ upper triangular unipotent,
diagonal, and lower triangular unipotent, respectively, then
    \begin{align}
&h_{i,i}  = 
\f{\det\bigl((g_{k,\ell})_{\,k,\ell\geq
i}\bigr)}{\det\bigl((g_{k,\ell})_{\,k,\ell>i}\bigr)} \quad \text{and} \quad n_{i,i+1} = \f{\det\Bigl((g_{k,\ell})_{\stackrel{\scriptstyle{k\geq i,\, k\neq
i+1}}{\ell>i }}\Bigr)}{\det\bigl((g_{k,\ell})_{k,\ell>i}\bigr)}\ . \label{superdiag}
    \end{align}
\end{lem}
This reduces parts (ii) and (iii) to computing the determinants of certain lower right subblocks of $wu$. We shall define the submatrices of $wn$ and $wu$ that are relevant to the formulas above. Let $1\leq i\leq r-1$.
\begin{definition}\label{defSi}
We let $S_i$ denote the lower right square submatrix of $wn$ of size $i$ whose rows consist of the $(r-i+1)$-th through the $r$-th rows of $wn$, and columns consist of the $(r-i+1)$-th through the $r$-th columns of $wn$.
\end{definition}
\begin{definition}\label{defEi}
We let $E_i$ denote the square submatrix of $wn$ of size $i$, whose rows consist of the $(r-i)$-th row and the $(r-i+2)$-th through the $r$-th rows of $wn$, and columns consist of the $(r-i+1)$-th through the $r$-th columns of $wn$.
\end{definition}
\begin{definition}\label{defTi}
We let $T_i$ and $F_i$ be the submatrices of $wu=R(wn)$ whose entries correspond to those of $S_i$ and $E_i$, respectively. 
\end{definition}
For $g=wu$, the formula \eqref{superdiag} can be written as \[ h_{i,i}=\frac{\det(T_{r-i+1})}{\det(T_{r-i})} \quad\text{and}\quad x_{i,i+1}=\frac{\det (F_{r-i})}{\det (T_{r-i})}.\]
\begin{definition}\label{Deltaidef}
For $1\leq i\leq r-1$, we let $\Delta_i(w)$ denote the product
$\prod\limits_{\substack{1\leq a<r-i+1\leq b\leq r, \\ (a,b)\in \Inv(\pi\i)}} n_{a,b}$.
\end{definition}
\begin{definition}\label{def:Lw}
For $1\leq i\leq r-1$, we let $L_i(w)$ denote the permutation matrix of size $(r-i)$ obtained from $w$ by removing the rows and columns which contain $1_j$ with $1\leq j\leq i$.
\end{definition}
We deduce part $(iii)$ of Theorem~\ref{biratlthm} from the following proposition. 
\begin{prop}\label{detTiprop}
For $1\leq i\leq r-1$, we have
\[\det \bigl(T_i\bigr)=\det\bigl( L_{r-i}(w)\bigr) \cdot \Delta_i(w).\]
\end{prop}
The proof is given throughout Sections~\ref{sec:parts23case1} and~\ref{sec:parts23case2}.
\begin{cor}[Part $(iii)$ of Theorem~\ref{biratlthm}] 
In the decomposition $wu=x\cdot b$, where $x\in N$ and $b\in B_{-}$, the $i$-th diagonal entry $b_{i,i}$ of $b$ satisfies \eqref{bdiag}.
\end{cor}
\begin{proof}
By Proposition~\ref{detTiprop}, we have
\begin{equation}\label{bii}
    b_{i,i}=\frac{\det(T_{r-i+1})}{\det(T_{r-i})}=\frac{\det\left(L_{i-1}(w)\right)}{\det\left(L_{i}(w)\right)}\cdot\frac{\Delta_{r-i+1}(w)}{\Delta_{r-i}(w)}.
\end{equation}
By Definition~\ref{Deltaidef}, we have
\[\Delta_{r-i+1}(w)=\prod\limits_{\substack{1\leq a<i\leq r \\ (a,i)\in \Inv(\pi\i)}}n_{a,i}\cdot\prod\limits_{\substack{1\leq a<i< b\leq r \\ (a,b)\in \Inv(\pi\i)}}n_{a,b}\]
and
\[\Delta_{r-i}(w)=\prod\limits_{\substack{1\leq a<i+1\leq b\leq r \\ (a,b)\in \Inv(\pi\i)}}n_{a,b}=\prod\limits_{\substack{1\leq a<i< b\leq r \\ (a,b)\in \Inv(\pi\i)}}n_{a,b}\cdot\prod\limits_{\substack{1\leq i< b\leq r \\ (i,b)\in \Inv(\pi\i)}}n_{i,b},
\]
thus
\begin{equation}\label{deltaratio}
    \frac{\Delta_{r-i+1}(w)}{\Delta_{r-i}(w)}=\frac{\prod_{ (a,i)\in \Inv(\pi\i)}n_{a,i}}{\prod_{(i,b)\in \Inv(\pi\i)}n_{i,b}}.
\end{equation}
Recall Definition~\ref{def:Lw}. After removing the rows and columns which contain the entries $1_1$, $1_2$, $\ldots$ , $1_{i-1}$, the matrix $1_i$ is in the $(\pi\i(i)-k,1)$-entry of $L_{i-1}(w)$, where $k$ is the number of entries $1_a$, $1\leq a<i-1$ which are above $1_i$. This implies
\begin{equation}\label{Liratio}
    L_{i-1}(w)=(-1)^{\pi\i(i)-k+1}L_{i}(w).
\end{equation}
Observe that the number $k$ equals
\begin{equation}\label{rewritek}
\begin{aligned}
  i-1-\abs{\{1\leq a<i\mid \pi\i(a)>\pi\i(i)\}}=i-1-\abs{\{1\leq a<i\mid (a,i)\in \Inv(\pi\i)\}}.   
\end{aligned}
\end{equation}
Applying \eqref{deltaratio}, \eqref{Liratio} and \eqref{rewritek} to \eqref{bii} yields \eqref{bdiag}.
\end{proof}
For the determinant of $F_i$, we introduce the notion of the level of a free variable $n_{\alpha}\in V_w$.
\begin{definition}\label{def:hnlevel}
For $1\leq a,b\leq r$, we let $h(a,b)$ denote the number of entries of $1$ in $wn$ which are below the $a$-th row and to the right of the $b$-th column, that is,
\[
    h(a,b)=\abs{\{1_j\in wn \mid j>b, \ \pi^{-1}(j)>a\}}.
\]
\end{definition}
If $n_{a,b}\in V_w$, then $h\bigl(\pi\i(a),b\bigr)$ equals the number of entries of $1$ which are below and to the right of $n_{a,b}$.
\begin{definition}\label{def:normdef}
For $n_{a,b}\in V_w$, we define the ``level" of $n_{a,b}$, denoted by $\norm{n_{a,b}}$, as
\[
    \norm{n_{a,b}}=\pi^{-1}(a)+b+h(\pi\i(a),b)-r.\]
\end{definition}
The level tells us how close a free variable $n_{a,b}$ is from the bottom right corner. See (3.43) in \cite{kimWhit} for an example that illustrates the level of free variables in $wn$.
\begin{definition}\label{def:bwi}
We let $B_{w}(i)\subseteq V_w$ denote the set of free variables of level $i$.
\end{definition}
The following lemma shows that $B_w(i)$ with $2\leq i\leq r$ partition $V_w$.
\begin{lem}\label{levelpartlem}
For every $n_{a,b}\in V_w$, we have $2\leq \norm{n_{a,b}}\leq \min\left(\pi\i (a),b\right)\leq r$.
\end{lem}
\begin{proof}
Consider the submatrix of $wn$ of size $(r-\pi\i(a))\times (r-b)$ whose rows consist of the $(\pi\i(a)+1)$-th through the $r$-th rows of $wn$ and columns consist of the $(b+1)$-th through the $r$-th columns of $wn$. Since there is at most one $1_j$ in each row or each column of the submatrix, we have $h(\pi\i(a),b)\leq \min(r-\pi\i (a),r-b)=r-\max(\pi\i (a),b)$. On the other hand, observe that there are $(\pi\i(a)-1)$ entries of $1$ above the row containing $n_{a,b}$, and $b-1$ entries of $1$ to the left of the column that contains $n_{a,b}$. All other entries of $1$ must be in the submatrix of size $(r-\pi\i(a))\times (r-b)$ defined above, hence $h(\pi\i(a),b)\geq r-(\pi\i(a)-1+b-1)$. The lemma follows.
\end{proof}
By Lemma~\ref{UDLlem}, the formula for $x_{i,i+1}$ in part (ii) of Theorem~\ref{biratlthm} is equivalent to the following proposition.
\begin{prop}\label{supdiagprop} For $1\leq i\leq r-1$, we have
\begin{equation}\label{supdiagformula}\frac{\det\left(F_{r-i}\right)}{\det\left(T_{r-i}\right)}=\sum_{n_{a,b}\in B_w(i+1)} \frac{1}{n_{a,b}}.
\end{equation}
\end{prop}
The proof is given throughout Sections~\ref{sec:parts23case1} and~\ref{sec:parts23case2}.
\begin{lem}\label{lem:basecase}
The Weyl group element $w=\left(
\begin{smallarray}{cc}
 0 & 1 \\
 1 & 0 \\
\end{smallarray}
\right)$ satisfy Propositions~ \ref{detTiprop} and~\ref{supdiagprop}. 
\end{lem}
\begin{proof}
 By \eqref{defR}, we have $u_{1,2}=n_{1,2}$, $wu=\Bigl(
\begin{smallarray}{cc}
 0 & 1 \\
 1 & n_{1,2} \\
\end{smallarray}
\Bigr)$, $T_1=(n_{1,2})$, and $F_1=(1)$.
Check that $\Delta_1(w)=n_{1,2}$, $L_1(w)=(1)$, and $\norm{n_{1,2}}=2$. The lemma follows.
\end{proof}
Using the above lemma as the base case, we use induction on $r$ to prove Propositions~\ref{detTiprop} and~\ref{supdiagprop}, assuming that they hold for every Weyl group element of $\GL(r')$ with $2\leq r'<r$.
\section{Parts (\textnormal{ii}) and (\textnormal{iii}) case 1: The bottom row of $T_i$ contains $1$}\label{sec:parts23case1}
In this section and the next, we prove Propositions~\ref{detTiprop} and~\ref{supdiagprop}. As discussed earlier, this proves parts (ii) and (iii) of Theorem~\ref{biratlthm}. We divide the proof into two cases. By Definition~\ref{defTi}, the bottom row of $T_i$ and $F_i$ contains $1_{\pi(r)}$ if $r-\pi(r)+1\leq i$. If $T_i$ and $F_i$ contain $1_{\pi(r)}$, the induction step proceeds by removing the bottom row of $T_i$ and $F_i$. If they do not contain $1_{\pi(r)}$, the induction step involves removing the rightmost column of $T_i$ and $F_i$, instead. \par 
The analysis of the former case is simpler because a free variable $n_\alpha$ in the bottom row does not appear in the birational map $R(n_\b)$ for any free variable $n_\b$ above the bottom row. This independence reduces the complexity of the analysis. The first paper \cite{kim2024pt1} of this series is devoted to the analysis of the latter case. In this section, we focus on the former case. Throughout this section, we fix a Weyl group element $w\in\Omega$ of $\GL(r)$.
\subsection{Change of indices after removing the bottom row}\label{sec:coirow}
Let $\widehat{w}$ denote the matrix obtained from $w$ by removing the bottom row and the $\pi(r)$-th column of $w$. Let $\widehat{\varphi}$ be the bijection from $\{1,\ldots,r\}\smallsetminus \{\pi(r)\}$ to $\{1,\ldots,r-1\}$ defined by 
\begin{equation}\label{phihat}
\widehat{\varphi}(i)=\begin{cases}
i & \text{if}\quad i<\pi(r) \\
i-1 & \text{if}\quad i>\pi(r),
\end{cases}
\quad
\widehat{\varphi}\i(i)=\begin{cases}
i & \text{if}\quad i<\pi(r) \\
i+1 & \text{if}\quad i\geq\pi(r).
\end{cases}\end{equation}
Check that the permutation $\widehat{\pi}\in \mathfrak{S}_{r-1}$ that corresponds to $\widehat{w}$ is given by
\begin{equation}\label{pihat}
\widehat{\pi}(i)=(\widehat{\varphi}\circ \pi)(i)=
\begin{cases}
\pi(i) & \text{if}\quad \pi(i)<\pi(r) \\
\pi(i)-1 & \text{if}\quad \pi(i)>\pi(r)
\end{cases}\end{equation}
for $1\leq i\leq r-1$.
\begin{definition}\label{def:wnwuhat}
    We let $\widehat{wn}$ and $\widehat{wu}$ denote the submatrices of $wn$ and $wu=R(wn)$ of size $(r-1)\times (r-1)$ obtained by removing the bottom row and the $\pi(r)$-th column of $wn$ and $wu$, respectively.
\end{definition}
\begin{definition}\label{def:whatnhatuhat}
We let $\what\nhat$ be the general matrix of $\what \bar{N}^{\what\i}$, where $\bar{N}\subset \GL(r-1)$ is the subgroup of unipotent upper triangular matrices, and $\nhat\in \bar{N}^{\what\i}$ is the unipotent upper triangular matrix with free variables $n_{\ahat,\bhat} \in V_{\what}$, as defined in \eqref{coords}. We let $\what\uhat$ be the matrix obtained from $\what \nhat$ by replacing the free variables $n_{\ahat,\bhat}$ in $\what \nhat$ with $u_{\ahat,\bhat}=R^{(\what)}(n_{\ahat,\bhat})$.
\end{definition}
The matrix $\widehat{wn}$ contains all of the free variables above the bottom row of $wn$. We extend the map $\phihat$ to the entries $n_{a,b}\in V_w$ above the bottom row of $wn$ by
\begin{equation}\label{phihatn}
    \widehat{\varphi}(n_{a,b})=n_{\widehat{\varphi}(a),\widehat{\varphi}(b)}=\begin{cases}
n_{a,b} & \text{if}\quad \pi(r)>b>a, \\
n_{a,b-1} &\text{if}\quad b>\pi(r)>a, \\
n_{a-1,b-1} &\text{if}\quad b>a>\pi(r).
\end{cases}
\end{equation}
Consider the example \eqref{ex:w3}. We have
\begin{equation}\label{ex:indiceslem}
 \widehat{w}=\left(
\begin{smallarray}{cccc}
 0 & 1 & 0 & 0 \\
 0 & 0 & 0 & 1 \\
 0 & 0 & 1 & 0 \\
 1 & 0 &0 & 0 \\
\end{smallarray}
\right),\quad \widehat{wn}=\left(
\begin{smallarray}{ccccc}
 0 & 1 & 0 & 0 \\
 0 & 0 & 0 & 1 \\
 0 & 0 & 1 & n_{4,5} \\
 1 & n_{1,2} & n_{1,4} & n_{1,5} \\
\end{smallarray}
\right), \quad \text{and}\quad \widehat{w}\nhat=\left(
\begin{smallarray}{cccc}
 0 & 1 & 0 & 0 \\
 0 & 0 & 0 & 1 \\
 0 & 0 & 1 & n_{3,4} \\
 1 & n_{1,2} & n_{1,3} & n_{1,4} \\
\end{smallarray}
\right).\end{equation}
Check that $\phihat$ maps $n_{1,2},n_{1,4},n_{1,5},n_{4,5}\in V_w$ to $n_{1,2},n_{1,3},n_{1,4},n_{3,4}\in V_{\what}$, respectively. Thus applying $\phihat$ to each free variables in $\widehat{wn}$ gives $\widehat{w}\nhat$. Accordingly, we further extend the map $\phihat$ elementwise to sets and matrices. For a set $A$ consisting of entries of $wn$ above the bottom row of $wn$, we define $\phihat(A)=\{\phihat (n_{a,b}) \mid n_{a,b}\in A\}$. Similarly, if $M$ is a submatrix of $wn$ which does not contain the bottom row and the $\pi(r)$-th column of $wn$, we let $\phihat(M)$ be the matrix of the same size which satisfies
\begin{equation}\label{phihatmat}
\left(\phihat(M)\right)_{i,j}=\begin{cases}
    (M)_{i,j} & \text{if } M_{i,j}=0 \text{ or } 1 \\
    \phihat(n_{\alpha}) & \text{if } M_{i,j}=n_{\alpha}\in V_w.
\end{cases}
\end{equation}
We also extend the map $\phihat$ to rational functions of the free variables in $V_w$ above the bottom row of $wn$ in an obvious way: we substitute every variable $n_{a,b}$ in the expression into $\widehat{\varphi}(n_{a,b})$. \par
If a free variable $n_{a,b}\in V_w$ is above the bottom row of $wn$, then by \eqref{defR}, none of the free variables in the bottom row appear in the rational function $R(n_{a,b})$, hence $(\phihat\circ R)(n_{a,b})$ is well defined. Therefore, we can further extend the map $\phihat$ entrywise to submatrices of $wu$. If $M$ is a submatrix of $wu$ which does not contain the bottom row and the $\pi(r)$-th column of $wu$, we define $\phihat(M)$ by
\begin{equation}\label{phihatmat1}
\left(\phihat(M)\right)_{i,j}=\begin{cases}
    (M)_{i,j} & \text{if } M_{i,j}=0 \text{ or } 1 \\
    \phihat(u_{\alpha}) & \text{if } M_{i,j}=u_{\alpha}=R(n_{\alpha}) \text{ for some } n_\alpha\in V_w.
\end{cases}
\end{equation}
To make a distinction between $w$ and $\what$, we use superscripts $(w)$ and $(\what)$ to previously defined notions. For example, we write $R(n_\a)=R^{(w)}(n_\a)$ if $n_\a$ is an element of $V_w$.
\begin{lem}\label{indiceslem}
We have
$\widehat{\varphi}(\widehat{wn})=\widehat{w}\nhat$, and
$V_{\widehat{w}}=\{\phihat(n_{a,b})\mid  n_{a,b}\in V_w, \, \pi\i(a)<r\}$.
Furthermore, $\widehat{\varphi}$ satisfies
$(\widehat{\varphi}\circ R^{(w)})(n_{a,b})=(R^{(\widehat{w})}\circ \widehat{\varphi})(n_{a,b})$
for any free variable $n_{a,b}\in V_w$ above the bottom row of $wn$. We have $\widehat{\varphi}(\widehat{wu})=\widehat{w}\uhat$.
\end{lem}
\begin{proof}
The statements on $\what\nhat$ and $V_{\what}$ are straightforward. Check that if $n_\a\in V_w$ is above the bottom row of $wn$, then the matrices $\phihat(M_{w} (n_\a))$ and $M_{\what}(\phihat(n_\a))$ are either identical or only differ by a zero column, with the latter occurring when $M_{w} (n_\a)$ contains the $\pi(r)$-th column of $wn$. The statements on $R^{(\what)}$ and $\what \uhat$ follow directly from this fact.
\end{proof}
Recall Definition~\ref{defTi}. We use the notation $T_i(w)$ and $F_i(w)$ to specify that they are the submatrices of $wu$. Similarly, $T_j(\what)$ and $F_j(\what)$ denote submarices of $\what\uhat$.
\begin{definition}\label{def:Tihat}
    For $r-\pi(r)+1\leq i\leq r-1$, we let $\widehat{T}_{i}(w)$ be the matrix obtained from $T_i(w)$ by removing the bottom row and the column which contains $1_{\pi(r)}$. Similarly, we let $\widehat{F}_{i}(w)$ be the matrix obtained from $F_i(w)$ by removing the bottom row and the column which contains $1_{\pi(r)}$.
\end{definition}
By Definition~\ref{def:wnwuhat}, $\widehat{T}_{i}(w)$ and $\widehat{F}_{i}(w)$ are submatrices of $\widehat{wu}$.
\begin{definition}\label{xkdef}
Let $1\leq i\leq r-1$. For $r-i\leq k\leq r$, we let $x_k$ denote the row vector of length $i$ which consists of the last $i$ entries of the $k$-th row of $wu$, that is, $x_k=(wu)_{k,r-i+1\leq j\leq r}$. For $r-i+1\leq j\leq r$, write $x_k(j)=(wu)_{k,j}=u_{\pi(k),j}$.
\end{definition}
Check that $T_i$ consist of the rows $x_{r-i+1},x_{r-i+2},\ldots,x_r$, and $F_i$ consists of the rows $x_{r-i},x_{r-i+2},x_{r-i+3},\ldots,x_r$.
\begin{definition}\label{def:xkhat}
Let $r-\pi(r)+1\leq i\leq r-1$. For $r-i\leq k\leq r-1$, we let $\widehat{x}_k$ denote the row vector of length $(i-1)$ obtained by removing the zero entry $x_k\bigl(\pi(r)\bigr)$ from the row $x_k$.
\end{definition}
Check that $\widehat{T}_{i}(w)$ consists of the rows $\widehat{x}_{r-i+1},\widehat{x}_{r-i+2},\ldots,\widehat{x}_{r-1}$, and $F_i$ consists of the rows $\widehat{x}_{r-i},\widehat{x}_{r-i+2},\widehat{x}_{r-i+3},\ldots,\widehat{x}_{r-1}$.
\begin{lem}\label{lem:TiFihatw}
For any $r-\pi(r)+1\leq i \leq r-1$, applying $\phihat$ to the submatrices $\widehat{T}_{i}(w)$ and $\widehat{F}_{i}(w)$ of $\widehat{wu}$ gives the submatrices $T_{i-1}(\what)$ and $F_{i-1}(\what)$ of $\what\uhat$, respectively.
\end{lem}
This is an immediate consequence of Lemma~\ref{indiceslem}. The proof is omitted.
\begin{lem}\label{lem:Tidefset}
Let $1\leq k\leq \pi(r)$. We have
\[
    \{n_{a,b}\in V_w \mid 1\leq a<k\leq b\leq r\} =\phihat\i\bigl(\{n_{\ahat,\bhat}\in V_{\what} \mid  1\leq \ahat<k\leq \bhat\leq r-1\}\bigr).
\]
\end{lem}
The proof is straightforward and therefore omitted.
\begin{prop}[Proposition~\ref{detTiprop} for case 1] \label{detTicase1}
    Assume that Proposition~\ref{detTiprop} holds for every Weyl group element $w'$ for $\GL(r')$ with $2\leq r'<r$. Let $r-\pi(r)+1\leq i \leq r-1$. We have
    \[\det \bigl(T_i\bigr)=\det\bigl( L_{r-i}(w)\bigr)\cdot \Delta_i(w).\]
\end{prop}
\begin{proof}
First, consider $i=1$. The inequality $r-\pi(r)+1\leq i$ implies that $\pi(r)=r$, hence the bottom right corner entry of $wn$ is $1_r$. Thus $T_1=(1)$ and $\Delta_i(w)$ is an empty product. The proposition follows trivially. Accordingly, we assume that $i\geq 2$. By the induction hypothesis, we have \begin{equation}\label{Case1:detTiind}
    \det \bigl(T_{i-1}(\what)\bigr)=\det\bigl( L_{r-i}(\what)\bigr) \cdot \Delta_{i-1}(\what),
\end{equation}
where
\[
   \Delta_{i-1}(\what)=\prod\limits_{\substack{1\leq \widehat{a}<(r-1)-(i-1)+1\leq \widehat{b}\leq r-1, \\ (\widehat{a},\widehat{b})\in \Inv(\pihat\i)}} n_{\widehat{a},\widehat{b}}=\prod\limits_{\substack{1\leq \widehat{a}<r-i+1\leq \widehat{b}\leq r-1, \\ (\widehat{a},\widehat{b})\in \Inv(\pihat\i)}} n_{\widehat{a},\widehat{b}}.
\]
Lemma~\ref{lem:TiFihatw} implies $\det\bigl(\widehat{T}_{i}(w)\bigr)=\phihat\i\bigl(\det(T_{i-1}(\what)\bigr)$, and Lemma~\ref{lem:Tidefset} with $k=r-i+1$ implies $\Delta_i(w)=\phihat\i\bigl(\Delta_{i-1}(\what)\bigr)$. Applying these to \eqref{Case1:detTiind} gives
\[
   \det\bigl(\widehat{T}_{i}(w)\bigr)=\det\bigl( L_{r-i}(\what)\bigr)\cdot \phihat\i\bigl(\Delta_{i-1}(\what)\bigr) =\det\bigl( L_{r-i}(\what)\bigr)\cdot \Delta_i(w).
\]
Also, since $1_{\pi(r)}$ is in the $(i,\pi(r)-r+i)$-entry of $T_i$, we have
\[
    \det \bigl(T_i(w)\bigr)=(-1)^{i+\pi(r)-r+i}\cdot \det\bigl(\widehat{T}_{i}(w)\bigr)=(-1)^{\pi(r)-r}\cdot\det\bigl( L_{r-i}(\what)\bigr)\cdot \Delta_i(w).
\]
Since $r-i+1\leq \pi(r)$, $1_{\pi(r)}$ remains after removing the $(r-i)$ leftmost columns from $w$. Observe that $1_{\pi(r)}$ is in the $(i,\pi(r)-r+i)$-entry of $L_{r-i}(w)$, and removing the bottom row and the column containing $1_{\pi(r)}$ from $L_{r-i}(w)$ gives $L_{r-i}(\what)$. It follows that \[ (-1)^{\pi(r)-r}\cdot\det\bigl( L_{r-i}(\what)\bigr)=\det\bigl( L_{r-i}(w)\bigr).\]
This finishes the proof.
\end{proof}
The final goal of this section is to prove Proposition~\ref{supdiagprop} for case 1 under the assumption that the proposition holds for every Weyl group element $w'$ of $\GL(r')$ with $2\leq r'<r$. Recall Definitions~\ref{def:hnlevel},~\ref{def:normdef}, and~\ref{def:bwi}.
\begin{lem}\label{lem:whatlevel}
Consider a free variable $n_{a,b}\in V_w$. If $n_{a,b}$ is in the bottom row of $wn$, that is, if $a=\pi(r)$, then $\norm{n_{a,b}}_{w}=b$.
If $n_{a,b}$ is above the bottom row, then
$\norm{n_{a,b}}_{w}= \norm{\phihat(n_{a,b})}_{\what}$.
\end{lem}
\begin{proof}
If $a=\pi(r)$, then $h(\pi\i(a),b)=0$, thus $\norm{n_{a,b}}_{w}=b$. On the other hand, if $n_{a,b}$ is above the bottom row of $wn$, then we have 
\[ \norm{\phihat(n_{a,b})}_{\what}=\norm{n_{\phihat(a),\phihat(b)}}_{\widehat{w}}=\pihat\i(\phihat(a))+\phihat(b)+h^{(\what)}\left(\pihat\i\bigl(\phihat(a)\bigr),\phihat(b)\right)-(r-1).
\]
It follows from \eqref{pihat} that $\pihat\i\bigl(\phihat(a)\bigr)=\pi\i(a)$. To prove the lemma, it suffices to show that
\begin{equation}\label{leveleq}
b+h^{(w)}(\pi\i(a),b)=\phihat(b)+h^{(\what)}\left(\pihat\i\bigr)\phihat(a)\bigr),\phihat(b)\right)+1.
\end{equation}
If $b<\pi(r)$, then $\phihat(b)=b$ by \eqref{phihat}. Also, in this case, $1_{\pi(r)}$ is below and to the right of $n_{a,b}$. Hence removing the bottom row and the $\pi(r)$-th column of $wn$ reduces the number of entries of $1$ that are below and to the right of $n_{a,b}$ by $1$. It follows that
\[
    h^{(w)}(\pi\i(a),b)=h^{(\what)}\left(\pihat\i\bigl(\phihat(a)\bigr),\phihat(b)\right)+1.
\]
Therefore, \eqref{leveleq} follows. On the other hand, if $b>\pi(r)$, then $\phihat(b)=b-1$, and $1_{\pi(r)}$ is to the left of $n_{a,b}$. In this case, removing the bottom row and the $\pi(r)$-th column of $wn$ does not affect the number of entries of $1$ which are below and to the right of $n_{a,b}$. It follows that
\[
    h^{(w)}(\pi\i(a),b)=h^{(\what)}\left(\pihat\i\bigl(\phihat(a)\bigr),\phihat(b)\right),
\]
and \eqref{leveleq} follows as well. This finishes the proof.
\end{proof}
\begin{prop}[Proposition~\ref{supdiagprop} for case 1]\label{supdiagcase1}
Assume that Proposition~\ref{supdiagprop} holds for every Weyl group element $w'$ of $\GL(r')$ with $2\leq r'<r$. Let $1\leq i\leq \pi(r)-1$, so that the bottom row of $T_{r-i}$ contains $1_{\pi(r)}$. We have
    \begin{equation}\label{Case1:detFiTi}
    \frac{\det\left(F_{r-i}\right)}{\det\left(T_{r-i}\right)}=\sum_{n_{a,b}\in B_w(i+1)} \frac{1}{n_{a,b}}.
\end{equation}
\end{prop}
\begin{proof}
First, consider the case $i=r-1$. Then the inequality $i\leq \pi(r)-1$ implies that $\pi(r)=r$, hence the bottom right corner entry of $wn$ is $1_r$. In this case, $F_1$ consists only of a single entry right above $1_{\pi(r)}$, which is zero. Thus $\det(F_1)/\det(T_1)=0$. On the other hand, by Lemma~\ref{levelpartlem} we have
$\norm{n_{a,b}}\leq \min(\pi\i (a),b)\leq r-1$, thus the set $B_w(r)$ is empty. The proposition follows. \par
From now on, we assume that $r-i\geq 2$. By the induction hypothesis, we have
\begin{equation}\label{Case1:detFiTiind}
    \frac{\det\left(F_{r-i-1}(\what)\right)}{\det\left(T_{r-i-1}(\what)\right)}=\sum_{n_{\ahat,\bhat}\in B_{\what}(i+1)} \frac{1}{n_{\ahat,\bhat}}.
\end{equation}
Recall Definition~\ref{def:Tihat}. Since $1_{\pi(r)}$ is in the $(r-i,\pi(r)-i)$-entry of $T_{r-i}(w)$ and $F_{r-i}(w)$, we have 
\begin{equation}\label{FTinduction0}
    \frac{\det\bigl(F_{r-i}(w)\bigr)}{\det\bigl(T_{r-i}(w)\bigr)}=   \frac{(-1)^{r-i+\pi(r)-i}\cdot \det\bigl(\widehat{F}_{r-i}(w)\bigr)}{(-1)^{r-i+\pi(r)-i}\cdot \det\bigl(\widehat{T}_{r-i}(w)\bigr)}=\frac{\det\bigl(\widehat{F}_{r-i}(w)\bigr)}{ \det\bigl(\widehat{T}_{r-i}(w)\bigr)}.  
\end{equation}
It follows from Lemma~\ref{lem:TiFihatw} that
\begin{equation}\label{FTinduction1}
    \frac{\det\bigl(F_{r-i}(w)\bigr)}{\det\bigl(T_{r-i}(w)\bigr)}=   \phihat\i\Bigl(\frac{\det\bigl(F_{r-i-1}(\what)\bigr)}{\det\bigl(T_{r-i-1}(\what)\bigr)}\Bigr).  
\end{equation}    
Applying (\ref{Case1:detFiTiind}) to the right hand side of \ref{FTinduction1}, we obtain
\[
   \frac{\det\bigl(F_{r-i}(w)\bigr)}{\det\bigl(T_{r-i}(w)\bigr)}=  \phihat\i\Bigl(\sum_{n_{\ahat,\bhat}\in B_{\what}(i+1)} \frac{1}{n_{\ahat,\bhat}}\Bigr)=\sum_{n_{\ahat,\bhat}\in B_{\what}(i+1)} \frac{1}{\phihat\i(n_{\ahat,\bhat})}.
\]
The proposition follows once we show that 
\begin{equation}\label{Case1:levelset}
B_w(i+1) = \phihat\i\bigl(B_{\what} (i+1)\bigr).\end{equation}
It folows from Lemma~\ref{lem:whatlevel} that a free variable $n_{\pi(r),b}$, which is in the bottom row, satisfies $\norm{n_{\pi(r),b}}_{w}=b>\pi(r)$. Since $\pi(r)\geq i+1$, we have $n_{\pi(r),b}\notin B_w(i+1)$. Since the elements of $V_w$ above the bottom row of $wn$ are in bijection with the elements of $V_{\what}$ by Lemma~\ref{indiceslem}, \eqref{Case1:levelset} follows immediately from Lemma~\ref{lem:whatlevel}.
\end{proof}

\section{Parts (\textnormal{ii}) and (\textnormal{iii}) case 2: The bottom row of $T_i$ does not contain $1$}\label{sec:parts23case2}
Throughout this section, we let $w$ be a fixed Weyl group element of $\GL(r)$. The submatrices $T_i$ and $F_i$ of $wu$ do not contain $1_{\pi(r)}$ if $i<r-\pi(r)+1$. In this case, the induction step involves removing the rightmost column and requires results from \cite{kim2024pt1}.
\subsection{Change of indices after removing the rightmost column}\label{sec:coicol}
Let $\widetilde{w}$ be the Weyl group element of size $(r-1)$ obtained from $w$ by removing the rightmost column and the $\pi\i(r)$-th row, which contains $1_r$. Let $\phitilde$ be a bijection from $\{1,\ldots,r\}\smallsetminus \{\pi\i(r)\}$ to $\{1,\ldots,r-1\}$ defined by
 \begin{equation}\label{phitilde}
\phitilde(i)=\begin{cases}
i & \text{if}\quad i<\pi\i(r) \\
i-1 & \text{if}\quad i>\pi\i(r),
\end{cases}
\quad \phitilde\i(i)=\begin{cases}
i & \text{if}\quad i<\pi\i(r) \\
i+1 & \text{if}\quad i\geq\pi\i(r).
\end{cases}\end{equation}
Check that the permutation $\pitilde\in \mathfrak{S}_{r-1}$ which corresponds to $\wtilde$ is given by
\begin{equation}\label{pitilde}
\widetilde{\pi}(i)=(\pi\circ \phitilde\i)(i)=
\begin{cases}
\pi(i) & \text{if}\quad i<\pi\i(r) \\
\pi(i+1) & \text{if}\quad i\geq\pi\i(r),
\end{cases}\end{equation}
for $1\leq i\leq r-1$. Here, the permutations in the symmetric group $\mathfrak{S}_{r-1}$ are tacitly considered as permutations lying in $\mathfrak{S}_{r}$ which fix $r$. The inverse of $\pitilde$ is given by
\begin{equation}\label{pitildeinv}
    \widetilde{\pi}\i(i)=(\phitilde\circ\pi\i)(i)=
\begin{cases}
\pi\i(i) & \text{if}\quad \pi\i(i)<\pi\i(r) \\
\pi\i(i)-1 & \text{if}\quad \pi\i(i)> \pi\i(r),
\end{cases}
\end{equation}
for $1\leq i\leq r-1$.
\begin{definition}\label{def:wnwutilde}
    We let $\wntilde$ and $\wutilde$ denote the submatrices of $wn$ and $wu$, respectively, obtained by removing the $\pi\i(r)$-th row and the rightmost column from $wn$ and $wu$.
\end{definition}
\begin{definition}\label{def:wtildenu}
We let $\wtilde\ntilde$ be the general matrix of $\wtilde \bar{N}^{\wtilde\i}$, where $\bar{N}\subset \GL(r-1)$ is the subgroup of unipotent upper triangular matrices, and $\ntilde\in \bar{N}^{\wtilde\i}$ is the unipotent upper triangular matrix with free variables $n_{\atilde,\btilde} \in V_{\wtilde}$. We let $\wtilde\utilde$ be the matrix obtained from $\wtilde \ntilde$ by replacing the free variables $n_{\atilde,\btilde}$ in $\wtilde \ntilde$ with $u_{\atilde,\btilde}=R^{(\wtilde)}(n_{\atilde,\btilde})$.
\end{definition}
The matrix $\widetilde{wn}$ contains all of the free variables in $wn$ which are to the left of the rightmost column. Consider again the example \eqref{ex:w3}. We have
\begin{equation}\label{ex:w3tilde}
\widetilde{w}=\left(
\begin{smallarray}{cccc}
 0 & 1 & 0 & 0 \\
 0 & 0 & 0 & 1 \\
 1 & 0 & 0 &0 \\
 0 & 0 & 1 &0 \\
\end{smallarray}
\right), \quad \widetilde{wn}=\left(
\begin{smallarray}{cccc}
 0 & 1 & 0 & 0 \\
 0 & 0 & 0 & 1 \\
 1 & n_{1,2} & 0 &n_{1,4} \\
 0 & 0 & 1 &n_{3,4} \\
\end{smallarray}
\right), \quad \text{and}\quad \widetilde{w}\ntilde=\left(
\begin{smallarray}{cccc}
 0 & 1 & 0 & 0 \\
 0 & 0 & 0 & 1 \\
 1 & n_{1,2} & 0 &n_{1,4} \\
 0 & 0 & 1 &n_{3,4} \\
\end{smallarray}
\right).\end{equation}
Observe that $\widetilde{wn}=\wtilde\ntilde$.
\begin{lem}[Lemma 6.99 of \cite{kimWhit}]\label{wntildewtilden}
We have $\widetilde{wn}=\wtilde\ntilde$, and $V_{\wtilde}=\{n_{a,b}\mid n_{a,b}\in V_w, \, b<r\}$.
\end{lem}
The proof is straightforward. A detailed proof is provided in \cite{kimWhit}.
\subsection{Reduction from $w$ to $\wtilde$} In this section, we write the determinants of $T_i(w)$ and $F_i(w)$ in terms of $T_{i-1}(\wtilde)$ and $F_{i-1}(\wtilde)$. First, let us record the results from \cite{kim2024pt1}. Recall Definition~\ref{xkdef} of $x_k$.
\begin{definition}\label{xLkdef}
   Let $1\leq i\leq r-1$. For $k\geq r-i$, $k\neq \pi\i(r)$, we let $x^L_k$ denote the row vector of length $(i-1)$, \[x^L_k=[x^L_k(r-i+1) \quad x^L_k(r-i+2) \quad \cdots \quad x^L_k(r-1)],\] where
    \begin{equation}\label{xLkdef1}
    x^L_k(j)=\begin{cases}
    R_L(n_{\pi(k),j}) & \text{if} \quad \bigl(\pi(k),j\bigr)\in \Inv(\pi\i) \\
    x_k(j) & \text{otherwise}
    \end{cases} 
    \end{equation}
    for $r-i+1\leq j\leq r-1$.
\end{definition}
\begin{definition}\label{TLdef} 
We let $T^L_i$ be the matrix of size $(i-1)\times (i-1)$ which equals
\begin{equation}\label{TiL}
\left(\begin{smallarray}{c}
 x^L_{r-i+2} \\ \vdots \\
 x^L_r
\end{smallarray}\right) \quad \text{or} \quad \left(\begin{smallarray}{c}
 x^L_{r-i+1} \\ \vdots \\  x^L_{r-i+s-1} \\ x^L_{r-i+s+1} \\ \vdots \\
 x^L_r
\end{smallarray}\right) \quad \text{with} \quad r-i+s=\pi\i(r),
\end{equation}
depending on whether $r-i+1>\pi\i(r)$ or $r-i+1\leq \pi\i(r)$, respectively.
\end{definition}
\begin{definition}\label{FLdef1}
   We let $F^L_i$ be the square matrix of size $(i-1)$ which equals
\begin{equation}\label{FiL}
       \left(\begin{smallarray}{c}
 x^L_{r-i+1} \\  x^L_{r-i+3} \\ \vdots \\ x^L_r
\end{smallarray}\right), \quad  \left(\begin{smallarray}{c}
 x^L_{r-i} \\  x^L_{r-i+3} \\ \vdots \\ x^L_r
\end{smallarray}\right), \quad \text{or}\quad \left(\begin{smallarray}{c}
  x^L_{r-i} \\ 
   x^L_{r-i+2}  \\ 
   \vdots \\ x^L_{r-i+s-1}  \\  x^L_{r-i+s+1}  \\
  \vdots\\ 
  x^L_r 
\end{smallarray}\right) \quad \text{with} \quad r-i+s=\pi\i(r),
  \end{equation}
depending on whether $r-i+1>\pi\i(r)$, $r-i+1=\pi\i(r)$, or $r-i+1<\pi\i(r)$, respectively. 
\end{definition} 
\begin{prop}[Proposition 1.16 of \cite{kim2024pt1}]\label{thm:rowcolop}
Assume that $1\leq i<r-\pi(r)+1$, so that the bottom row of $T_i$ and $F_i$ does not contain the entry of $1$ in the bottom row of $wu$.
    \begin{enumerate}
        \item[(i)] We have
$\det (T_i)=(-1)^{s+i} \cdot u_{\pi(r-i+s),r}\cdot \det( T_i^L)$, where $s$ is such that $x_{r-i+s}$ is the row of $T_i$ that contains the highest nonzero entry of the rightmost column of $T_i$.
\item[(ii)] We have
\[\det (F_i)=\begin{cases}
    \frac{1}{n_{\pi(r-i+1),r}}\det(T_i)-\frac{(-1)^{i}}{n_{\pi(r-i+1),r}}\cdot u_{\pi(r-i+2),r}\cdot \det (F^L_i) & \text{ if } r-i+1>\pi\i(r) \\
   (-1)^{i}\cdot u_{\pi(r-i+2),r}\cdot \det (F^L_i) & \text{ if } r-i+1=\pi\i(r) \\
   (-1)^{r-\pi\i(r)}\cdot \det (F^L_i) & \text{ if } r-i+1<\pi\i(r).
\end{cases}\]
    \end{enumerate}
\end{prop}
We write $T_i^L=T_i^L(w)$ and $F_i^L=F_i^L(w)$ to specify their associated Weyl group element. We shall describe how $T_i^L(w)$ and $F_i^L(w)$ are related to submatrices of $\wtilde\utilde$. For this, the following lemma is used.
Recall Definition~\ref{def:RLR1R2}.
\begin{lem}[Corollary 3.14 of \cite{kim2024pt1}]\label{RLwRwtilde}
Let $n_{a,b}\in V_w$ be a free variable that is not in the rightmost column of $wn$. We have
\begin{equation}\label{RLwRwtilde0}
R^{(w)}_L(n_{a,b})=\begin{cases}
R^{(\widetilde{w})}(n_{a,b}) & \text{if} \quad \pi\i(b)<\pi\i(a)<\pi\i(r), \\
-n_{a,r}\cdot R^{(\widetilde{w})}(n_{a,b}) & \text{if} \quad \pi\i(b)<\pi\i(r)<\pi\i(a), \\
\frac{n_{a,r}}{n_{b,r}}\cdot R^{(\widetilde{w})}(n_{a,b}) & \text{if} \quad \pi\i(r)<\pi\i(b)<\pi\i(a).
\end{cases}
\end{equation}
\end{lem}
Recall Definitions~\ref{defSi} and~\ref{defEi} of $S_i$ and $E_i$. We use the notation $S_i(w)$ and $E_i(w)$ to specify the associated Weyl group element.
\begin{definition}\label{ykdef}
For $k\geq r-i$, we let $y_k$ denote the row vector of length $i$ which consists of the last $i$ entries of the $k$-th row of $wn$, that is, $y_k=(wn)_{k,r-i+1\leq j\leq r}$. For $r-i+1\leq j\leq r$, we write $y_k(j)=(wn)_{k,j}=n_{\pi(k),j}$.
\end{definition}
\begin{definition}\label{yktildedef}
For $k\geq r-i$, we let $\widetilde{y}_k$ denote the row vector of length $(i-1)$ which is obtained by removing the rightmost entry $y_k(r)$ of $y_k$. For $r-i+1\leq j\leq r-1$, we write $\widetilde{y}_k(j)=y_k(j)$.
\end{definition}
The following lemma is analogous to Lemma~\ref{lem:TiFihatw}.
\begin{lem}\label{SwEwtildelem}
If $1\leq i<r-\pi(r)+1$, then $T^L_i(w)$ and $F^L_i(w)$ coincide with those obtained from $S_{i-1}(\widetilde{w})$ and $E_{i-1}(\widetilde{w})$, respectively, by replacing each entry $n_\alpha \in V_{\widetilde{w}}$ with $R^{(w)}_L(n_\alpha)$.
\end{lem}
\begin{proof}
By Lemma~\ref{wntildewtilden}, $S_{i-1}(\wtilde)$ and $E_{i-1}(\wtilde)$ can be regarded as submatrices of $\wntilde$. By Definition~\ref{phitilde}, the row number $j$ of the rows of $wn$ with $j\neq \pi\i(r)$ becomes $\phitilde(j)$ after removing the $\pi\i(r)$-th row. Thus the $j$-th row of $\wntilde$ equals the $\phitilde\i(j)$-th row of $wn$ with the rightmost entry removed. It follows that
\begin{equation}\label{Sitildephi}
   S_{i-1}(\wtilde)= \left(\begin{smallarray}{c}
 \widetilde{y}_{\phitilde\i(r-i+1)} \\ \widetilde{y}_{\phitilde\i(r-i+2)} \\ \vdots \\ \widetilde{y}_{\phitilde\i(r-1)}
\end{smallarray}\right) \eqand E_{i-1}(\wtilde)= \left(\begin{smallarray}{c}
 \widetilde{y}_{\phitilde\i(r-i)} \\ \widetilde{y}_{\phitilde\i(r-i+2)} \\ \vdots \\ \widetilde{y}_{\phitilde\i(r-1)}
\end{smallarray}\right).
\end{equation}
By (\ref{phitilde}), we have $\phitilde\i(j)=j$ for $j<\pi\i(r)$ and $\phitilde\i(j)=j+1$ for $j\geq \pi\i(r)$. Observe from \eqref{TiL} and \eqref{FiL} that replacing the rows $\widetilde{y}_l$ of $S_{i-1}(\wtilde)$ and $E_{i-1}(\wtilde)$ with $x^L_l$ gives the matrices $T_i^L(w)$ and $F_i^L(w)$, respectively. The lemma follows by Definitions~\ref{xLkdef} and~\ref{ykdef}, together with the fact that $x_k(j)=(wu)_{k,j}=(wn)_{k,j}=y_k(j)$ whenever $y_k(j)\notin V_w$.
\end{proof}
Finally, we prove that $T^L_i(w)$ and $F^L_i(w)$ are equal to $T_{i-1}(\widetilde{w})$ and $F_{i-1}(\widetilde{w})$, respectively, each multiplied on the left and right by diagonal matrices.
\begin{definition}\label{rowcolscale}
    For $1\leq a,\, b\leq r-1$, we define the row scaling factors as
\begin{equation}\label{rowscale}
s_{\text{row}}(a)=\begin{cases}
n_{\widetilde{\pi}(a),r} & \text{if} \quad \pi\i\left(\widetilde{\pi}(a)\right)>\pi\i(r) \\ 
-1 & \text{otherwise},
\end{cases}
\end{equation}
and the column scaling factors as
\begin{equation}\label{colscale}
s_{\text{col}}(b)=\begin{cases}
n_{b,r} & \text{if} \quad \pi\i(b)>\pi\i(r) \\
-1 & \text{otherwise.}
\end{cases}
\end{equation}
\end{definition}
\noindent Observe from \eqref{pitilde} that $(\pi\i\circ \pitilde)(a)=\phitilde\i(a)$ for all $1\leq a\leq r-1$. It is straightforward to check that the following formula is equivalent to \eqref{rowscale}:
\begin{equation}\label{rowscaleequiv}
s_{\text{row}}(a)=\begin{cases}
n_{\pi(a+1),r} & \text{if} \quad a\geq\pi\i(r) \\ 
-1 & \text{otherwise}.
\end{cases}
\end{equation}
\begin{lem}\label{RLwscale}
Let $n_{a,b}\in V_w$ be a free variable that is not in the rightmost column of $wn$. We have
\[R^{(w)}_L(n_{a,b})=\frac{s_{\text{row}}\bigl(\widetilde{\pi}\i(a)\bigr)}{s_{\text{col}}(b)}R^{(\widetilde{w})}(n_{a,b}).\] 
\end{lem}
\begin{proof}
It follows from Definition~\ref{rowcolscale} that for any $a$ and $b$ such that $1\leq a, b\leq r-1$ and $\pi\i(b)<\pi\i(a)$ we have
\[
\begin{aligned}
\frac{s_{\text{row}}\left(\widetilde{\pi}\i(a)\right)}{s_{\text{col}}(b)}=\begin{cases}
1 & \text{if} \quad \pi\i(b)<\pi\i(a)<\pi\i(r) \\
-n_{a,r} & \text{if} \quad \pi\i(b)<\pi\i(r)<\pi\i(a) \\
\frac{n_{a,r}}{n_{b,r}} & \text{if} \quad \pi\i(r)<\pi\i(b)<\pi\i(a).
\end{cases}
\end{aligned}
\]
By Lemma~\ref{RLwRwtilde}, the lemma follows.
\end{proof}
\begin{lem}\label{LTdiag}
    The matrix $T^L_i(w)$ satisfies
\begin{equation}\label{LTscale}
T^L_i(w)=\left(\begin{smallmatrix}
 s_{\text{row}}(r-i+1) &   \\
  & \ddots & \\
  & & s_{\text{row}}(r-1)
 \end{smallmatrix}
 \right)\cdot T_{i-1}(\widetilde{w})\cdot \left(\begin{smallmatrix}
 s_{\text{col}}(r-i+1) &  \\
  & \ddots & \\
  & & s_{\text{col}}(r-1)
 \end{smallmatrix}
 \right) ^{-1},
\end{equation} 
and the matrix $F^L_i(w)$ satisfies
\begin{equation}\label{LFscale}
F^L_i(w)=\left(\begin{smallmatrix}
 s_{\text{row}}(r-i) & & &\\
 & s_{\text{row}}(r-i+2) & & \\
  & & \ddots & \\
 & & & s_{\text{row}}(r-1)
 \end{smallmatrix}
 \right)\cdot F_{i-1}(\widetilde{w})\cdot \left(\begin{smallmatrix}
 s_{\text{col}}(r-i+1) &  \\
  & \ddots & \\
  & & s_{\text{col}}(r-1)
 \end{smallmatrix}
 \right) ^{-1}.
\end{equation}  
\end{lem}
\begin{proof}
By Lemma~\ref{SwEwtildelem}, it suffices to show that the left and right multiplication of $T_{i-1}(\widetilde{w})$ and $F_{i-1}(\widetilde{w})$ by the two diagonal matrices transforms the entries of $T_{i-1}(\widetilde{w})$ and $F_{i-1}(\widetilde{w})$ of the form $R^{(\wtilde)}(n_{\a})$, where $n_\a\in V_{\wtilde}$, into $R_L^{(w)}(n_{\a})$, while leaving all other entries unchanged. Since the $(a,b)$-entry of $T_{i-1}(\widetilde{w})$ is $(\widetilde{w}\utilde)_{r-i+a,r-i+b}=u_{\pitilde(r-i+a),r-i+b}$, the $(a,b)$-entry of the matrix on the right hand side of \eqref{LTscale} equals
\begin{equation}\label{diagTLform}
    \frac{s_{\text{row}}(r-i+a)}{s_{\text{col}}(r-i+b)}\cdot(\wtilde\utilde)_{r-i+a,r-i+b}.
\end{equation}
Similarly, the $(a,b)$-entry of the matrix $F_{i-1}(\widetilde{w})$ and the matrix on the right hand side of \eqref{LFscale} are given by
\begin{equation}
   (\wtilde\utilde)_{r-i+a',r-i+b} \eqand \frac{s_{\text{row}}(r-i+a')}{s_{\text{col}}(r-i+b)}\cdot(\wtilde\utilde)_{r-i+a',r-i+b},
\end{equation}
respectively, where $a'=0$ if $a=1$ and $a'=a$ if $a\geq 2$. The entries of $T_{i-1}(\widetilde{w})$ and $F_{i-1}(\widetilde{w})$ are either $0$, $1$, or $R^{\widetilde{w}}(n_\a)$ for some $n_\a \in V_{\widetilde{w}}$. Obviously, zero entries are not affected by the multiplication of diagonal matrices. If $(\widetilde{w}\utilde)_{r-i+a,r-i+b}=1$, then $\widetilde{\pi}(r-i+a)=r-i+b$, and by Definition~\ref{rowcolscale} this implies $s_{\text{row}}(r-i+a)=s_{\text{col}}(r-i+b)$. Hence the entries of $1$ in $T_{i-1}(\widetilde{w})$ and $F_{i-1}(\widetilde{w})$ remain unchanged as well. Finally, if $n_{\pitilde(r-i+a),r-i+b}\in V_{\wtilde}$, then by Lemma~\ref{RLwscale} we have
\[
 \frac{s_{\text{row}}(r-i+a)}{s_{\text{col}}(r-i+b)}R^{(\widetilde{w})}(n_{\widetilde{\pi}(r-i+a),r-i+b})=R^{(w)}_L(n_{\widetilde{\pi}(r-i+a),r-i+b}).
\]
This completes the proof.
\end{proof}
\begin{cor}\label{lem:detTTwtilde}
Let $1\leq i< r-\pi(r)+1$, and let $x_{r-i+s}$ be the row of $T_i$ which contains the highest nonzero entry of the rightmost column. We have
\[\det \left(T_i(w)\right)=(-1)^{s+i}\cdot u_{\pi(r-i+s),r}\cdot\frac{\prod\limits_{r-i+1\leq a\leq r-1} s_{\text{row}}(a)}{\prod\limits_{r-i+1\leq b\leq r-1} s_{\text{col}}(b)}\cdot \det \left(T_{i-1}(\widetilde{w})\right).\]
\end{cor}
\begin{proof}
This is a direct consequence of Proposition~\ref{thm:rowcolop} and Lemma~\ref{LTdiag}.
\end{proof}
\subsection{Proof of Proposition~\ref{detTiprop}}\label{sec:detTiproppf}
In this section, we prove Proposition~\ref{detTiprop} for case 2 under the assumption that the proposition holds for every Weyl group element of $\GL(r')$ with $2\leq r'<r$. Together with Lemma~\ref{lem:basecase} and Proposition~\ref{detTicase1}, this completes the proof of Proposition~\ref{detTiprop}.
\begin{lem}\label{lem:Rrm}
If $n_{\pi(a),r}\in V_w$, then
$u_{\pi(a),r}=\prod\limits_{\pi\i(r)<j\leq a} n_{\pi(j),r}$ and $\frac{u_{\pi(a-1),r}}{u_{\pi(a),r}}=\frac{1}{n_{\pi(a),r}}$.
\end{lem}
\begin{proof}
The first equation is a direct consequence of the fact that $\mathcal{P}(n_{\pi(a),r})$ consists of a single vertical path from $n_{\pi(a),r}$ to $1_r$. The second equation follows directly from the first. 
\end{proof}
For the next lemma, recall Definition~\ref{Deltaidef}. 
\begin{lem}\label{Deltailem}
For any $2\leq i\leq r-1$ we have \[\frac{\Delta_i(w)}{\Delta_{i-1}(\widetilde{w})}= \prod_{\substack{\pi\i(r)<j\leq r,\\ 1\leq \pi(j)<r-i+1}} n_{\pi(j),r}.\]
\end{lem}
\begin{proof}
Since $n_{a,r}$ is a free variable if and only if $\pi\i(a)>\pi\i(r)$, we can write $\Delta_i(w)$ as
\[
   \prod\limits_{\substack{1\leq a<r-i+1\leq b\leq r-1, \\ (a,b)\in \Inv(\pi\i)}} n_{a,b}\cdot \prod\limits_{\substack{1\leq a<r-i+1, \\ \pi\i(a)>\pi\i(r)}} n_{a,r}.
\]
Writing $a$ in the second product as $\pi(j)$, we obtain
\[
\Delta_i(w)=\prod\limits_{\substack{1\leq a<r-i+1\leq b\leq r-1, \\ (a,b)\in \Inv(\pi\i)}} n_{a,b}\cdot \prod\limits_{\substack{\pi\i(r)<j\leq r, \\ 1\leq \pi(j)<r-i+1}} n_{\pi(j),r}.
\]
By Lemma~\ref{wntildewtilden}, the set of free variables $n_{a,b}\in V_w$ with $b\leq r-1$ equals the set $V_{\wtilde}$. Thus the first product equals
$\Delta_{i-1}(\widetilde{w})$. The lemma follows.
\end{proof}
\begin{prop}[Proposition~\ref{detTiprop} for case 2]\label{detTicase2}
Suppose that Proposition~\ref{detTiprop} holds for every Weyl group element $w'$ of $\GL(r')$ with $2\leq r'<r$. Let $1\leq i< r-\pi(r)+1$. We have
\[\det \bigl(T_i(w)\bigr)=\det\bigl( L_{r-i}(w)\bigr)\cdot \Delta_i(w).\] 
\end{prop}
\begin{proof}
First, assume that $i=1$. By Definition~\ref{defTi}, $T_1$ consists of a single entry $u_{\pi(r),r}$. It follows from the inequality $r-i+1>\pi(r)$ that $n_{\pi(r),r}\in V_w$, and from Definition~\ref{Deltaidef} and Lemma~\ref{lem:Rrm} that $u_{\pi(r),r}=\Delta_1(w)$. The proposition follows. \par
From now on, assume that $i\geq 2$. By the induction hypothesis, we have
\[\det \left(T_{i-1} (\widetilde{w})\right)=\det L_{r-i}(\widetilde{w})\cdot \Delta_{i-1}(\widetilde{w}).
\]
Let $x_{r-i+s}$ be the row of $T_i$ which contains the highest nonzero entry of the rightmost column. By Corollary~\ref{lem:detTTwtilde} and Lemma~\ref{Deltailem}, we deduce
\begin{equation}\label{detTired1}
     \det \left(T_i(w)\right)
     =(-1)^{s+i}\cdot \det L_{r-i}(\widetilde{w}) \cdot \frac{u_{\pi(r-i+s),r}\cdot\prod\limits_{r-i+1\leq a\leq r-1} s_{\text{row}}(a)}{\prod\limits_{\substack{\pi\i(r)<j\leq r,\\ 1\leq \pi(j)<r-i+1}} n_{\pi(j),r}\cdot\prod\limits_{r-i+1\leq b\leq r-1} s_{\text{col}}(b)}\cdot\Delta_i(w).
\end{equation}
We shall show that the fractional term equals $\pm 1$ after cancellation. Since the entry $1_r$ is in the $\pi\i(r)$-th row of $wu$, $1_r$ is an entry of $T_i(w)$ precisely when $r-i+1\leq \pi\i(r)$. It follows that $r-i+s=\max\bigl(r-i+1,\pi\i(r)\bigr)$.
\par
First, if $\pi\i(r)<r-i+1$, so that $1_r$ is not in $T_i(w)$, then we have $s=1$ and $n_{\pi(r-i+1),r}\in V_w$. By Lemma~\ref{lem:Rrm}, we have
\[
u_{\pi(r-i+s),r}=\prod\limits_{\pi\i(r)<j\leq r-i+s}n_{\pi(j),r}.\]
Also, by \eqref{rowscaleequiv}, the product $\prod_{r-i+1\leq a\leq r-1} s_{\text{row}}(a)$ equals
\[\prod_{r-i+1\leq a_1< r-i+s} s_{\text{row}}(a_1)\cdot\prod_{r-i+s\leq a_2\leq r-1}s_{\text{row}}(a_2)=(-1)^{s-1}\cdot \prod_{r-i+s\leq a_2\leq r-1} n_{\pi(a_2+1),r}.\]
We conclude that
\begin{equation}\label{detTinum}
 u_{\pi(r-i+s),r}\cdot\prod\limits_{r-i+1\leq a\leq r-1} s_{\text{row}}(a)=(-1)^{s-1}\cdot\prod\limits_{\pi\i(r)<j\leq r} n_{\pi(j),r}.   
\end{equation}
This gives a simplified expression for the numerator. \par
By a similar approach, we simplify the denominator. Let $S_{\text{col}}$ be the set defined by
\begin{equation}\label{colset1}
S_{\text{col}}=\{r-i+1\leq j\leq r-1 \mid \pi\i(j)<\pi\i(r) \},\end{equation}
and let $x$ be the number of elements the set $S_{\text{col}}$. By \eqref{colscale}, we have
\[
\begin{aligned}
    \prod_{r-i+1\leq b\leq r-1} s_{\text{col}}(b)&=\prod\limits_{\substack{r-i+1\leq b\leq r-1, \\ \pi\i(b)< \pi\i(r)}} s_{\text{col}}(b)\cdot \prod\limits_{\substack{r-i+1\leq b\leq r-1, \\ \pi\i(b)>\pi\i(r)}}  s_{\text{col}}(b)=(-1)^x\cdot\prod\limits_{\substack{r-i+1\leq b\leq r-1, \\ \pi\i(r)<\pi\i(b)}}n_{b,r}.
\end{aligned}
\]
Substituting $b$ with $\pi(j)$, we obtain
\[\prod\limits_{r-i+1\leq b\leq r-1} s_{\text{col}}(b)  =(-1)^x\cdot \prod\limits_{\substack{\pi\i(r)<j\leq r, \\ r-i+1\leq\pi(j)\leq r-1}}n_{\pi(j),r}.\] 
Thus
\begin{equation}\label{detTidenom}
\prod\limits_{\substack{\pi\i(r)<j\leq r,\\ 1\leq \pi(j)<r-i+1}} n_{\pi(j),r}\cdot\prod\limits_{r-i+1\leq b\leq r-1} s_{\text{col}}(b)=(-1)^x \cdot \prod\limits_{\substack{\pi\i(r)<j\leq r, \\ 1\leq\pi(j)\leq r-1}}n_{\pi(j),r}.
\end{equation}
The condition $1\leq \pi(j)\leq r-1$ is automatically satisfied whenever $\pi\i(r)<j\leq r$. By \eqref{detTinum} and \eqref{detTidenom}, we conclude that the fractional term on the right hand side of \eqref{detTired1} equals $(-1)^{s-1+x}$. The proposition follows once we show that
\begin{equation}\label{detLw}
    (-1)^{i+1+x}\cdot \det L_{r-i}(\widetilde{w})=\det L_{r-i}(w).
\end{equation}
Recall from Definition~\ref{def:Lw} that $L_{r-i}(w)$ is obtained from $w$ by removing the leftmost $(r-i)$ columns and the $(r-i)$ rows which contain $1_1,\ 1_2, \ \ldots \ ,\ 1_{r-i}$ from $w$. For $j<\pi\i(r)$, the $j$-th row is not removed if and only if $\pi(j)>r-i$. By the definition \eqref{colset1} of $S_{\text{col}}$, the $j$-th row of $w$ remains in $L_{r-i}(w)$ precisely when $\pi(j)$ is in the set $S_{\text{col}}$. Since $x$ is the number of elements of the set $S_{\text{col}}$, we deduce that precisely $x$ rows above $1_r$ remain in $L_{r-i}(w)$. Therefore, $1_r$ is in the $(x+1,i)$-entry of $L_{r-i}(w)$. Since removing the rightmost column and the row containing $1_r$ from $L_{r-i}(w)$ gives $L_{r-i}(\widetilde{w})$, \eqref{detLw} follows.
\end{proof}

\subsection{Proof of Proposition~\ref{supdiagprop}}\label{sec:supdiagproppf} In this section, we prove Proposition~\ref{supdiagprop} under the assumption that the proposition holds for every Weyl group element $w'$ of $\GL(r')$ with $2\leq r'<r$. Together with Lemma~\ref{lem:basecase} and Proposition~\ref{supdiagcase1}, this completes the proof of Proposition~\ref{supdiagprop}.
Assume that $\pi(r)<i+1$, so that $T_{r-i}$ and $F_{r-i}$ do not contain $1_{\pi(r)}$.
\begin{lem}\label{lem:samelevel}
Let $n_{a,b}\in V_w$. If $n_{a,b}$ is in the rightmost column of $wn$, then
    $\norm{n_{a,b}}_{w}=\pi\i(a)$.
If $b<r$, then $\norm{n_{a,b}}_{w}= \norm{n_{a,b}}_{\wtilde}$.
\end{lem}
We omit the proof because this lemma is analogous to Lemma~\ref{lem:whatlevel} for case 1, and the proof follows the same approach.
\begin{cor}\label{cor:levelset}
For $1\leq j\leq r-1$, we have
\begin{equation}\label{levelsetcase2}
    B_{w}(j)=\begin{cases}
        B_{\wtilde}(j) & \text{if} \quad j\leq \pi\i(r) \\
        B_{\wtilde}(j) \cup \{n_{\pi(j),r}\} & \text{if} \quad j> \pi\i(r).
    \end{cases}
    \end{equation}
\end{cor}
\begin{proof}
By Lemma~\ref{wntildewtilden}, the set of free variables in $V_w$ that are not in the rightmost column of $wn$ equals the set $V_{\wtilde}$. By Lemma~\ref{lem:samelevel}, $n_{a,b}\in V_w$ with $b<r$ is in $B_w(j)$ if and only if $n_{a,b}\in B_{\wtilde}(j)$, and a free variable $n_{a,r}\in V_w$ is in $B_w(j)$ if and only if $a=\pi(j)$. Since $n_{\pi(j),r}\in V_w$ precisely when $j>\pi\i(r)$, the lemma follows.
\end{proof}
\begin{prop}[Proposition~\ref{supdiagprop} for case 2]\label{supdiagcase2}
Let $\pi(r)< i+1 \leq r$. Suppose that Proposition~\ref{supdiagprop} holds for every Weyl group element $w'$ of $\GL(r')$ with $2\leq r'<r$. We have
\[\frac{\det\left(F_{r-i}(w)\right)}{\det\left(T_{r-i}(w)\right)}=\sum_{n_{a,b}\in B_w(i+1)} \frac{1}{n_{a,b}}.\]
\end{prop}
\begin{proof}
First, consider the case $i=r-1$. Then the inequality $\pi(r)<i+1$ implies that $n_{\pi(r),r}\in V_w$. By Definition~\ref{defTi}, we have $T_1=(u_{\pi(r),r})$ and $F_1=(u_{\pi(r-1),r})$, hence by Lemma~\ref{lem:Rrm} we have $\det(F_{1})/\det(T_{1})=\frac{1}{n_{\pi(r),r}}$. Also, it follows from Lemma~\ref{levelpartlem} that a free variable $n_{a,b}\in V_w$ satisfies $\norm{n_{a,b}}=r$ if and only if $a=\pi(r)$ and $b=r$. This proves the proposition for the case $i+1=r$. \par
From now on, we assume that $i+1< r$, so that $T_{r-i}$ and $F_{r-i}$ have size at least $2$. Let $x_{i+s}$ be the row of $T_{r-i}$ that contains the highest nonzero entry of the rightmost column of $T_{r-i}$. Then we have $i+s=\max\bigl(i+1,\pi\i(r)\bigr)$, hence by Proposition~\ref{thm:rowcolop} with $i$ replaced by $(r-i)$ we have
\begin{equation}\label{FTdetratiocase2}
    \begin{aligned}
 \frac{\det\left(F_{r-i}(w)\right)}{\det\left(T_{r-i}(w)\right)}=\begin{cases}
\frac{1}{n_{\pi(i+1),r}}+\frac{1}{n_{\pi(i+1),r}}\cdot\frac{u_{\pi(i+2),r}}{u_{\pi(i+1),r}}\cdot\frac{\det\left(F^L_{r-i}(w)\right)}{\det\left(T^L_{r-i}(w)\right)} & \text{if} \quad  i+1>\pi\i(r)     \\
 - \frac{u_{\pi(i+2),r}}{u_{\pi(i+1),r}}\cdot\frac{\det\left(F^L_{r-i}(w)\right)}{\det\left(T^L_{r-i}(w)\right)} & \text{if} \quad  i+1=\pi\i(r) \\
 \frac{\det\left(F^L_{r-i}(w)\right)}{\det\left(T^L_{r-i}(w)\right)} & \text{if} \quad  i+1<\pi\i(r).
 \end{cases}        
    \end{aligned}
\end{equation}
We have $u_{\pi(i+2),r}=u_{\pi(i+1),r}n_{\pi(i+2),r}$ by Lemma~\ref{lem:Rrm}. Also, it is a direct consequence of Lemma~\ref{LTdiag} that we have
\begin{equation}\label{LFLTratio}
\frac{\det\left(F^L_{r-i}(w)\right)}{\det\left(T^L_{r-i}(w)\right)}=\frac{s_{\text{row}}(i)}{s_{\text{row}}(i+1)}\cdot\frac{\det \left(F_{r-i-1}(\widetilde{w})\right)}{\det \left(T_{r-i-1}(\widetilde{w})\right)},    
\end{equation}
where, by \eqref{rowscaleequiv}, we have
\begin{equation}\label{srowratio}
 \frac{s_{\text{row}}(i)}{s_{\text{row}}(i+1)}=\begin{cases}
\hfil \frac{n_{\pi(i+1),r}}{n_{\pi(i+2),r}} & \text{if } i+1>\pi\i(r)     \\
\hfil -\frac{1}{n_{\pi(i+2),r}} & \text{if } i+1=\pi\i(r) \\
\hfil  1 & \text{if } i+1<\pi\i(r).
 \end{cases}        
\end{equation}
Applying \eqref{LFLTratio} and \eqref{srowratio} to \eqref{FTdetratiocase2} gives
\begin{equation}\label{FTdetratiocase22}
    \begin{aligned}
 \frac{\det\left(F_{r-i}(w)\right)}{\det\left(T_{r-i}(w)\right)}=\begin{cases}
\frac{1}{n_{\pi(i+1),r}}+\frac{\det \left(F_{r-i-1}(\widetilde{w})\right)}{\det \left(T_{r-i-1}(\widetilde{w})\right)} & \text{if} \quad  i+1>\pi\i(r)     \\
\frac{\det \left(F_{r-i-1}(\widetilde{w})\right)}{\det \left(T_{r-i-1}(\widetilde{w})\right)} & \text{if} \quad  i+1\leq\pi\i(r).
 \end{cases}        
    \end{aligned}
\end{equation}
By the induction hypothesis, this implies
\begin{equation}\label{FTdetratiocase2-2}
    \begin{aligned}
 \frac{\det\left(F_{r-i}(w)\right)}{\det\left(T_{r-i}(w)\right)}=\begin{cases}
\frac{1}{n_{\pi(i+1),r}}+\sum\limits_{n_{a,b}\in B_{\wtilde}(i+1)} \frac{1}{n_{a,b}} & \text{if} \quad  i+1>\pi\i(r)     \\
\sum\limits_{n_{a,b}\in B_{\wtilde}(i+1)} \frac{1}{n_{a,b}} & \text{if} \quad  i+1\leq\pi\i(r).
 \end{cases}        
    \end{aligned}
\end{equation}
By Corollary~\ref{cor:levelset}, this finishes the proof of the proposition.
\end{proof}

\section{Part (\textnormal{i}) of Theorem~\ref{biratlthm}}\label{sec:part1}
Fix a Weyl group element $w\in \Omega$. In this section, we prove part (i) of Theorem~\ref{biratlthm} by showing that each $n_\a$ can be written as ratios of determinants of submatrices of $wu$. For each $n_{x,y}\in V_w$, we assign a submatrix $K(n_{x,y})$ of $wu$, which is the largest invertible matrix whose bottom left corner entry is $u_{x,y}$, and then compute its determinant. This can be viewed as an extension of Proposition~\ref{detTiprop}: for a free variable $n_{\pi(r),y}$ in the bottom row of $wn$, $T_{r-y+1}$ is the largest invertible submatrix of $wu$ whose bottom left corner entry is $u_{\pi(r),y}$.
\begin{definition}\label{def:Kcol}
    For each free variable $n_{x,y}\in V_w$, let
\begin{equation}\label{Kcoldef}
K_{\text{col}}(n_{x,y})=\{j\in\Z \mid y\leq j\leq r \text{ and } \pi\i(j)\leq  \pi\i(x)\}.
\end{equation}
\end{definition}
Since $n_{x,y}\in V_w$, we have $\pi\i(y)<\pi\i(x)$, hence $y\in K_{\text{col}}(n_{x,y})$. Thus $K_{\text{col}}$ is nonempty. Observe that for $y\leq j\leq r$, $j\in K_{\text{col}}(n_{x,y})$ if and only if $1_j$ is above the row which contains $n_{x,y}$. Such a $j$ satisfies $n_{x,j}\in V_w$.
\begin{definition}\label{Kndef}
For $n_{x,y}\in V_w$ such that $K_{\text{col}}(n_{x,y})=\{y_1,\ldots,y_t\}$ with $y=y_1<\cdots<y_t$, we let $K(n_{x,y})$ denote the $t\times t$ submatrix of $wu$ that consists of the $(\pi\i(x)-t+1)$-th through the $\pi\i(x)$-th rows of $wu$, and the $y_1$-th, $y_2$-th, $\ldots$ , and $y_t$-th columns of $wu$.
\end{definition}
Observe that the bottom left corner entry of $K(n_{x,y})$ is $(wu)_{\pi\i(x),y}=u_{x,y}$. 
\begin{definition}\label{kappandef}
We let $\kappa(n_{x,y})$ be the $t\times t$ submatrix of $w$, consisting of the rows and columns that contain $1_{y_1},\ldots 1_{y_t}$.
\end{definition}
\begin{definition}\label{def:Kdet}
For $n_{x,y}\in V_w$, we let
\begin{equation}\label{eq:Kdet}
    K_{\text{det}}(n_{x,y})=\{n_{a,b}\in V_w \mid 1\leq a<y\leq b\leq r, \quad \pi\i(a)\leq \pi\i (x)\},
\end{equation}    
\end{definition}
In case more than one Weyl group element is considered, we write $K^{(w)}_{\text{col}}(n_{x,y})$, $K_w(n_{x,y})$, $\kappa_w(n_{x,y})$, and $K^{(w)}_{\text{det}}(n_{x,y})$ to specify that they are associated to $w$. \par
We use induction to compute the determinant of $K(n_{x,y})$. For this, we use the map $\phihat$ defined in \eqref{phihat}. Consider again the example \eqref{ex:w3} and let $n_{1,2}\in V_{w}$.
 \[wn=\left(
\begin{smallarray}{ccccc}
 0 & 1_2 & 0 & 0 & 0 \\
\cline{2-2}\cline{4-5}
 0 & \multicolumn{1}{|c|}{0} & 0 & \multicolumn{1}{|c|}{0} & \multicolumn{1}{c|}{1_5} \\
 0 & \multicolumn{1}{|c|}{0} & 0 & \multicolumn{1}{|c|}{1_4} & \multicolumn{1}{c|}{n_{4,5}} \\
 1_1 & \multicolumn{1}{|c|}{n_{1,2}} & 0 & \multicolumn{1}{|c|}{n_{1,4}} & \multicolumn{1}{c|}{n_{1,5}} \\
 \cline{2-2}\cline{4-5}
 0 & 0 & 1_3 & n_{3,4} & n_{3,5} \\
\end{smallarray}
\right) \leadsto K_w(n_{1,2})=\left(\begin{smallarray}{ccc}
0 & 0 & 1_5 \\
0 & 1_4 & u_{4,5} \\
u_{1,2} & u_{1,4} & u_{1,5} \\
\end{smallarray}\right), \quad \kappa_w(n_{1,2})=\left(\begin{smallarray}{ccc}
    1_2 & 0 & 0 \\
    0 & 0 & 1_5 \\
    0 & 1_4 & 0 \\ 
\end{smallarray}\right).\]
Check that $K^{(w)}_{\text{col}}(n_{1,2})=\{2,4,5\}$ and $K^{(w)}_{\text{det}}(n_{1,2})=\{n_{1,2},n_{1,4},n_{1,5}\}$. On the other hand, from \eqref{ex:indiceslem}, we obtain
\[
    \widehat{w}\nhat=
\left(
\begin{smallarray}{cccc}
 0 & 1_2 & 0 & 0 \\
\cline{2-4}
 0 & \multicolumn{1}{|c}{0} & \multicolumn{1}{|c|}{0} & \multicolumn{1}{c|}{1_4} \\
 0 & \multicolumn{1}{|c}{0} & \multicolumn{1}{|c|}{1_3} & \multicolumn{1}{c|}{n_{3,4}} \\
 1_1 & \multicolumn{1}{|c}{n_{1,2}} & \multicolumn{1}{|c|}{n_{1,3}} & \multicolumn{1}{c|}{n_{1,4}} \\
 \cline{2-4}
\end{smallarray}
\right)
\leadsto
K_{\what}(n_{1,2})=\left(\begin{smallarray}{ccc}
0 & 0 & 1_4 \\
0 & 1_3 & u_{3,4} \\
u_{1,2} & u_{1,3} & u_{1,4} \\
\end{smallarray}\right), \quad \kappa_{\what}(n_{1,2})=\left(\begin{smallarray}{ccc}
    1_2 & 0 & 0 \\
    0 & 0 & 1_4 \\
    0 & 1_3 & 0 \\ 
\end{smallarray}\right).
\]
Also, we have
$K^{(\what)}_{\text{col}}(n_{1,2})=\{2,3,4\}$ and $K^{(\what)}_{\text{det}}(n_{1,2})=\{n_{1,2},n_{1,3},n_{1,4}\}$. It follows from Lemma~\ref{indiceslem} that $\phihat\bigl(K_w(n_{1,2})\bigr)=K_{\what}(n_{1,2})$.
\begin{lem}\label{lem:Kphihat}
   Let $n_{x,y}\in V_w$ be a free variable above the bottom row of $wn$. We have
   \[\phihat \left(K_w(n_{x,y})\right)= K_{\what}(n_{\phihat(x),\phihat(y)}), \quad \phihat (K^{(w)}_{\text{det}}(n_{x,y}))=K^{(\what)}_{\text{det}}(n_{\phihat(x),\phihat(y)}),\] and $\kappa_w(n_{x,y})=\kappa_{\what}(n_{\phihat(x),\phihat(y)})$.
\end{lem}
We omit the proof, as it follows directly from a straightforward application of Lemma~\ref{indiceslem}. See (7.45) through (7.52) of \cite{kimWhit} for a detailed proof.
\begin{lem}\label{detKnlem}
Let $n_{x,y}\in V_w$. We have
\[\det\bigl(K(n_{x,y})\bigr)=\det\bigl(\kappa(n_{x,y})\bigr)\cdot \prod_{n_{a,b}\in K_{\text{det}(n_{x,y})}}n_{a,b}.\]
\end{lem}
\begin{proof} 
First, consider a free variable $n_{\pi(r),y}\in V_w$ in the bottom row of $wn$. We have $K_{\text{col}}(n_{\pi(r),y})=\{y,y+1,\ldots, r\}$, $K(n_{\pi(r),y})=T_{r-y+1}$ by Definition~\ref{defTi}, $\kappa(n_{\pi(r),y})=L_{y-1}(w)$ by Definition~\ref{def:Lw}, and
\[\prod_{n_{a,b}\in K_{\text{det}(n_{\pi(r),y})}}n_{a,b}=\Delta_{r-y+1}(w)\]
by Definition~\ref{Deltaidef}. The lemma follows from Proposition~\ref{detTiprop}. For $n_{x,y}\in V_w$ above the bottom row of $wn$, we use induction on the size $r$ of $w$. The base case of $r=2$ follows directly from Lemma~\ref{lem:basecase}. Assume that the lemma holds for every Weyl group element $w'$ of $\GL(r')$ with $2\leq r'\leq r-1$. By the induction hypothesis and Lemma~\ref{lem:Kphihat}, we have
\[\begin{aligned}
\phihat\left(\det\bigl(K_w(n_{x,y})\bigr)\right)&=\det\left(K_{\what}(n_{\phihat(x),\phihat(y)})\right)=\det\bigl(\kappa_{\what}(n_{\phihat(x),\phihat(y)})\bigr)\cdot \prod\limits_{n_{\ahat,\bhat}\in K^{(\what)}_{\text{det}}(n_{\phihat(x),\phihat(y)})} n_{\ahat,\bhat} \\
&=
\det\bigl(\kappa_w(n_{x,y})\bigr) \cdot\prod_{n_{a,b}\in K_{\text{det}(n_{x,y})}}\phihat(n_{a,b}).
\end{aligned}\]
Applying $\phihat\i$ to both sides finishes the proof.
\end{proof}
Part (i) of Theorem~\ref{biratlthm} follows once we show that every free variable $n_{x,y}$ can be written as the ratio of products of the determinants of $K(n_{\alpha})$. This shows that the inverse of the rational map $R$ exists as a rational function.
\begin{lem}\label{Knsetlem}
Let $n_{x,y}\in V_w$. We have the following:
\begin{enumerate}
    \item [(i)] If there is no free variable above $n_{x,y}$ in the $y$-th column, then every element in the set $K_{\text{det}}(n_{x,y})$ other than $n_{x,y}$ is to the right of the $y$-th column of $wn$. 
    \item [(ii)] If there is a free variable in the $y$-th column above $n_{x,y}$, then let $n_{x_1,y}$ be the free variable that is located in the lowest row among all such variables. To simplify notations, write $K_{\text{det}}(n_{x,y})=K_{\text{det}}$ and $K_{\text{det}}(n_{x_1,y})=K^{(1)}_{\text{det}}$. We have
    $K^{(1)}_{\text{det}}\subseteq K_{\text{det}}$ and $n_{x,y}\in (K_{\text{det}}\smallsetminus K^{(1)}_{\text{det}})$. Also, every element of $(K_{\text{det}}\smallsetminus K^{(1)}_{\text{det}})$ other than $n_{x,y}$ is in the $\pi\i(x)$-th row, to the right of the $n_{x,y}$.
\end{enumerate}
\end{lem}
\begin{proof} 
Part (i) is straightforward. For part (ii), observe that by the choice of $n_{x_1,y}$ there is no free variable between $n_{x,y}$ and $n_{x_1,y}$. We have $\pi\i(x_1)<\pi\i(x)$, and this implies $ K^{(1)}_{\text{det}}\subseteq K_{\text{det}}$. By Definition~\ref{def:Kdet} we have
\begin{equation}\label{Kdet1minusK1det1}
    K_{\text{det}}\smallsetminus K^{(1)}_{\text{det}}=\bigl\{n_{a,b}\in V_w \mid 1\leq a<y\leq b\leq r, \quad \pi\i (x_1)<\pi\i(a)\leq \pi\i (x)\bigr\}.
\end{equation}
It is clear that $n_{x,y}\in ( K_{\text{det}}\smallsetminus K^{(1)}_{\text{det}})$. Suppose that $n_{a,b}\in (K_{\text{det}}\smallsetminus K^{(1)}_{\text{det}})$ and $n_{a,b}\neq n_{x,y}$. We have $\pi\i(y)<\pi\i(x_1)<\pi\i(a)$ and $a<y$, thus $n_{a,y}$ is a free variable in the $y$-th column, below $n_{x_1,y}$, and in or above the $\pi\i(x)$-th row. Since there is no free variable between $n_{x,y}$ and $n_{x_1,y}$, we must have that $n_{a,y}=n_{x,y}$. Thus $n_{a,b}$ is in the $\pi\i(x)$-th row. Since $n_{a,b}\neq n_{x,y}$, the inequality $y\leq b$ implies that $n_{a,b}$ is to the right of $n_{x,y}$. This proves (ii).
\end{proof}
\begin{proof}[Proof of part (i) of Theorem~\ref{biratlthm}]
The proof amounts to showing that every free variable $n_{x,y}$ can be written as a rational function of the coordinates $u_{i,j}$. We use descending induction on the column index $y$ of the free variables $n_{x,y}\in V_w$. For the base case $y=r$, we have $n_{\pi(x),r}=\frac{u_{\pi(x),r}}{u_{\pi(x-1),r}}$ for $n_{\pi(x),r}\in V_w$ by Lemma~\ref{lem:Rrm}, so the assertion holds. Now, let $1\leq y<r$ and assume that every free variable to the right of the $y$-th column can be written as a rational function of $u_{i,j}$. Consider $n_{x,y}\in V_w$. Write $K_{\text{det}}=K_{\text{det}}(n_{x,y})$. By Lemma~\ref{detKnlem}, we have
\[
n_{x,y}=\det\bigl(\kappa(n_{x,y})\bigr)\cdot\det\bigl(K(n_{x,y})\bigr)\cdot \prod\limits_{n_{a,b}\in (K_{\text{det}} \smallsetminus \{n_{x,y}\})} \frac{1}{n_{a,b}}.
\]
If there is no free variable in the $y$-th column of $wn$ above $n_{x,y}$, then by part (i) of Lemma~\ref{Knsetlem}, every free variable that appears in the product is to the right of the $y$-th column. By the induction hypothesis, we conclude that $n_{x,y}$ is a rational function of $u_{i,j}$. If there is a free variable in the $y$-th column of $wn$ above $n_{x,y}$, let $n_{x_1,y}$ be the free variable that is located in the lowest row among all such variables. Write $ K^{(1)}_{\text{det}}= K_{\text{det}}(n_{x_1,y})$. By part (ii) of Lemma~\ref{Knsetlem}, we have $K^{(1)}_{\text{det}}\subseteq K_{\text{det}}$ and $n_{x,y}\in (K_{\text{det}}\smallsetminus K^{(1)}_{\text{det}})$. By Lemma~\ref{detKnlem}, we have
\[
n_{x,y}=\frac{\det\bigl(\kappa(n_{x_1,y})\bigr)}{\det\bigl(\kappa(n_{x,y})\bigr)}\cdot \frac{\det\bigl(K(n_{x,y})\bigr)}{\det\bigl(K(n_{x_1,y})\bigr)} \cdot \prod\limits_{\substack {n_{a,b}\in (K_{\text{det}}\smallsetminus K^{(1)}_{\text{det}}) \\ n_{a,b}\neq n_{x,y}}} \frac{1}{n_{a,b}}.
\]
By part (ii) of Lemma~\ref{Knsetlem}, every free variable that appears in the product is to the right of $n_{x,y}$. By the induction hypothesis, we conclude that $n_{x,y}$ is a rational function of $u_{i,j}$. This finishes the proof.
\end{proof}

\section{Part (\textnormal{iv}) of Theorem~\ref{biratlthm}}\label{sec:part4}
In this section, we prove part (iv) of Theorem~\ref{biratlthm}. Fix a Weyl group element $w\in \Omega$.
\begin{lem}\label{lem:part4partial}
Let $n_{x,y}\in V_w$. We have \begin{equation}\label{eq:part4step1}
    \frac{\partial u_{x,y}}{\partial n_{x,y}}=(-1)^{\abs{D(n_{x,y})}-1}\frac{\prod\limits_{\substack {d> y \\ (x,d)\in \Inv(\pi\i)}}n_{x,d}\prod \limits_{\substack{\pi\i(y)< l< \pi\i(x) \\ (\pi(l),y)\in \Inv(\pi\i)}}n_{\pi(l),y}}{\prod\limits_{(y,f)\in \Inv(\pi\i)}n_{y,f}}.
\end{equation} 
\end{lem}
\begin{proof}
First, we claim that there exists a unique set of disjoint paths $\mathbf{p}$ in ${\mathcal P}(n_{x,y})$ that contains $n_{x,y}$. Recall from Definition~\ref{mwna} that $n_{x,y}$ is in the bottom left corner entry of $M_w(n_{x,y})$. Thus for a path $p$ in the set $\mathbf{p}\in {\mathcal P}(n_{x,y})$ to pass through $n_{x,y}$, $p$ must start from the bottom origin $\gamma(1_x)$ and end at the top destination $1_y$, forming an L-shaped path. Also, for every element of $O(n_{x,y})$ other than $\gamma(1_x)$, there is a unique element of $D(n_{x,y})$ other than $1_y$ located in the same row. Consequently, all other paths in $\mathbf{p}$ must be horizontal. This proves our claim. \par
Since $n_{x,y}$ does not appear in the denominator in \eqref{defR} with $n_\alpha=n_{x,y}$, it follows from our claim that
\[\frac{\partial u_{x,y}}{\partial n_{x,y}}=(-1)^{\abs{D(n_{x,y})}-1}\frac{1}{n_{x,y}}\frac{u(\mathbf{p})}{\prod\limits_{1_\mu\in D(n_{x,y})}\rho(1_\mu)}.\]
The L-shaped path in $\mathbf p$ passes through the variables in the bottom row and the leftmost column of $M_w(n_{x,y})$. The other paths are horizontal, passing through every free variable containing $1_\mu\in D(n_{x,y})$ other than the top destination $1_y$. Therefore
\[
u(\mathbf{p})=n_{x,y} \cdot \prod\limits_{\substack {d> y \\ n_{x,d}\in V_w}} n_{x,d}\cdot\prod \limits_{\substack{\pi\i(y)< l< \pi\i(x) \\ n_{\pi(l),y}\in V_w}}n_{\pi(l),y} \cdot \prod\limits_{1_\mu\in D(n_{x,y})\smallsetminus \{1_y\}} \rho(1_\mu).\]
The first and second products represent the horizontal and vertical segments of the $L$-shaped path, respectively. The product in the denominator of the right hand side of \eqref{eq:part4step1} equals $\rho(1_y)$. This finishes the proof.
\end{proof}
For the following proposition, from which part (iv) of Theorem~\ref{biratlthm} follows as a corollary, we introduce an ordering $\sqsupset$ to the indices $(i,j)\in \Inv(\pi\i)$ by the lexicographic order on the pair $(-j,\pi\i(i))$. That is, \begin{equation}\label{ordering2}
(i_1,j_1)\sqsupset (i_2,j_2) \quad \text{if} \quad
j_1<j_2,\quad \text{or if} \quad j_1=j_2 \quad \text{and}\quad \pi\i(i_1)>\pi\i(i_2).\end{equation} We extend the ordering in the obvious way to the variables $n_{i,j}$, that is, 
$n_{i_1,j_1}\sqsupset n_{i_2,j_2}$ if $(i_1,j_1)\sqsupset (i_2,j_2)$. Note that this ordering is different from the ordering $\succ$ defined in \eqref{ordering}. 
\begin{prop}\label{lem:Jacobian}
Write $\Inv(\pi\i)=\{\alpha_1,\alpha_2,\ldots ,\alpha_d\}$, where $\alpha_1\sqsupset \cdots\sqsupset \alpha_d$. The Jacobian matrix of $R$ under the ordering $\sqsupset$, with $(i,j)$-entry given by $\f{\partial u_{\alpha_i}}{\partial n_{\alpha_j}}$, is upper triangular, and its determinant is \[(-1)^{t_w} \prod_{(i,j)\in \Inv(\pi\i)} n_{i,j}^{j-i-1},\] where $t_w=\sum_{(x,y)\in \Inv(\pi\i)} \abs{D(n_{x,y})}-d$.
\end{prop}
\begin{proof}
Observe that $n_{i_1,j_1}\sqsupset n_{i_2,j_2}$ if $n_{i_1,j_1}$ is either to the left of $n_{i_2,j_2}$, or in the same column but below $n_{i_2,j_2}$. Thus if $n_{i_1,j_1}\sqsupset n_{i_2,j_2}$ then $n_{i_1,j_1}$ is not in the submatrix $M_w(n_{i_2,j_2})$, hence $n_{i_1,j_1}$ does not appear in the rational function $u_{i_2,j_2}=R(n_{i_2,j_2})$. Therefore the Jacobian matrix of $R$ under this ordering $\sqsupset$ is upper triangular, and its determinant equals the product $\prod_{(x,y)\in \Inv(\pi\i)} \frac{\partial u_{x,y}}{\partial n_{x,y}}$. \par
Let $n_{i,j}\in V_w$. Recall Definition~\ref{def:hnlevel}. First, we count the number of free variables $n_{x,y}$ such that $n_{i,j}$ appears in the denominator of the formula \eqref{eq:part4step1}, and call such number $A(i,j)$. This happens whenever the top destination of $n_{x,y}$ is $1_i$, that is, whenever $y=i$. Hence $A(i,j)$ equals the number of free variables in the $i$-th column. Since $1_{i}$ is in the $\pi\i(i)$-th row, there are $\bigl(r-\pi\i(i)\bigr)$ entries below $1_i$. Observe that an entry $n_{a,i}$ below $1_i$ is zero precisely when $1_a$ is below and to the right of $1_i$, and is a free variable otherwise. It follows that $A(i,j)=r-\pi\i(i)-h\bigl(\pi\i(i),i\bigr)$.
\par
Next, we count the number of free variables $n_{x,y}$ such that $n_{i,j}$ appears in the numerator of the formula \eqref{eq:part4step1}, and call such number $B(i,j)$. Recall that the first and the second products in the numerator represent the horizontal and the vertical segments of the L-shaped path that passes through $n_{x,y}$, respectively. Let $B_1(i,j)$ and $B_2(i,j)$ denote the number of $n_{x,y}\in V_w$ such that $n_{i,j}$ appears in the first and the second product in the numerator, respectively, so that $B(i,j)=B_1(i,j)+B_2(i,j)$. By a computation similar to that of $A(i,j)$, we have
\[B_1(i,j)=\bigl(r-i-h(\pi\i(i),i)\bigr)-\bigl(r-j-h(\pi\i(i),j)\bigr),\]
and $B_2(i,j)=r-\pi\i(i)-h(\pi\i(i),j)$. We conclude that the power of $n_{i,j}\in V_w$ in the product $\prod_{(x,y)\in \Inv(\pi\i)} \frac{\partial u_{x,y}}{\partial n_{x,y}}$ equals $B(i,j)-A(i,j)=j-i-1$. The product of the sign factors in \eqref{eq:part4step1} over all $n_{x,y}$ equals $(-1)^{t_w}$. The lemma follows.
\end{proof}
Part (iv) of Theorem~\ref{biratlthm} follows immediately from Proposition~\ref{lem:Jacobian}, by taking the absolute value of the determinant of the Jacobian.

\section{Part (\textnormal{v}) of Theorem~\ref{biratlthm}}\label{sec:part5}
In this section, we prove part (v) of Theorem~\ref{biratlthm}. Fix a Weyl group element $w\in \Omega$ of $\GL(r)$ with the corresponding permutation $\pi$. We use results from \cite{kim2024pt1}.
\begin{definition}\label{def:Vwabj}
For integers $a,b,j$ such that $1\leq j\leq r$ and $1\leq a<b\leq r$, we let
\begin{equation}\label{eq:Vwabj}
    V_w(a,b,j)=\bigl\{n_{x,y}\in V_w \mid a\leq \pi\i(y)<\pi\i(x)\leq b, \ j\leq y \bigr\}.
\end{equation}
That is, $n_{x,y}\in V_w(a,b,j)$ if $n_{x,y}\in V_w$ is in or above the $b$-th row, below the $a$-th row, in or to the right of the $j$-th column, and its top destination $1_y$ is in or below the $a$-th row.
\end{definition}
For the next definition, recall Definition~\ref{def:gamma1} of $\gamma(1_{\mu})$.
\begin{definition}\label{def:DabjOabj}
For any integers $a$, $b$, and $j$ satisfying 
\begin{equation}\label{eq:abjineq}
   1\leq a<b\leq r \eqand \pi(b)< j\leq \pi(a),
\end{equation} we let $D_w(a,b,j)$ denote the set of entries of $1$ which are in or to the right of the $j$-th column, above the $b$-th row and in or below the $a$-th row of $wn$, that is,
\begin{equation}\label{eq:Dabj}
    D_w(a,b,j)=\{1_{\mu}\in wn \mid a\leq \pi\i(\mu)<b, \ j\leq \mu \}.
\end{equation}
Also, we let 
\begin{equation}\label{eq:Oabj}
O_w(a,b,j)=\{\gamma(1_{\pi(b)})\}\cup \gamma \left(D_w(a,b,j)\right)\smallsetminus\{\gamma(1_{\pi(a)})\}.\end{equation}
\end{definition}
In Section 5 of \cite{kim2024pt1}, we introduce the notion of a ``$P_L$-decomposable set $D$ with respect to $d$" and ``the set $O$ of origins for a $P_L$-decomposable set $D$", and derive a decomposition formula \cite[Lemma 5.15]{kim2024pt1} for the polynomial $P_L(O\to D)$. We say that a set $D$ of entries of $1$ in $wn$ is $P_L$-decomposable with respect to $d$ if every element of $D$ is above and to the right of $1_{\pi(d)}$, and define the corresponding set of origins by $O=\{\gamma(1_{\pi(d)})\}\cup \gamma(D) \smallsetminus \{\gamma(1_\lambda)\}$, where $1_\lambda$ is the element of $D$ which is located in the highest row among all elements of $D$. Observe that the inequalities \eqref{eq:abjineq} imply that $D_w(a,b,j)$ is a $P_L$-decomposable set with respect to $b$. Also, it follows from the above definition that $1_{\pi(a)}\in D_w(a,b,j)$ and that every other element of $D_w(a,b,j)$ is below the $a$-th row, thus $O_w(a,b,j)$ is the corresponding set of origins. Thus we can use the decomposition formula.
\begin{lem}\label{lem:nxyabjform}
Let $n_{x,y}\in V_w$. Then the inequalities \eqref{eq:abjineq} are satisfied with $a$, $b$, and $j$ replaced by $\pi\i(y)$, $\pi\i(x)$, and $y$, respectively, and we have
\[D(n_{x,y})=D_w(\pi\i(y),  \pi\i(x), y) \eqand O(n_{x,y})=O_w(\pi\i(y), \pi\i(x), y). \]
    Also, if $k$ satisfies $1_{\pi(k)}\in D(n_{x,y})$ and $\pi(k)\neq y$ then the inequalities \eqref{eq:abjineq} are satisfied with $a$, $b$, and $j$ replaced by $k$, $\pi\i(x)$, and $(y+1)$, respectively, and we have
    \[
D_{\downarrow,k}(n_{x,y}) =D_w(k,\pi\i(x),y+1), \eqand
O_{\downarrow,k+1}(n_{x,y}) =O_w(k,\pi\i(x),y+1). \]
\end{lem}
We omit the proof, as it follows immediately from Definitions~\ref{Ddef}, \ref{Odef}, and~\ref{ODup}. For the following lemma, recall Definition~\ref{D1}.
\begin{lem}\label{lem:R1ndecomp}
Let $n_{\pi(d),j}\in V_w$, and write
$O_1(\npdj)=\{n_{\pi(\d_1),r},\ldots,n_{\pi(\d_s),r}\}$. We have
\[R_1(n_{\pi(d),j})=\sum_{m=1}^{s} (-1)^{t_m} \cdot n_{\pi(\delta_m),r} \cdot u_{\pi(\d_m-1),j}\cdot \frac{P_L\bigl(O_w(\d_m,d,j+1)\to D_w(\d_m,d,j+1)\bigr)}{\prod\limits_{1_\mu\in D_w(\d_m,d,j+1)}\rho(1_\mu)},\]
where $t_{m}=\abs{D_w(\d_m,d,j+1)}$.
\end{lem}
\begin{proof}
It is proved in \cite[Lemma 5.10]{kim2024pt1} that the above equation with $O_w(\d_m,d,j+1)$ and $D_w(\d_m,d,j+1)$ replaced by $O_{\downarrow,\delta_m+1}(n_{\pi(d),j})$ and $D_{\downarrow,\delta_m}(n_{\pi(d),j})$, respectively, holds. It follows from Definitions~\ref{O1} and~\ref{D1} that each $1\leq i\leq s$ satisfies $1_{\pi(\d_i)}\in D(\npdj)$ and $\pi(\d_i)\neq j$. The lemma follows from Lemma~\ref{lem:nxyabjform}.
\end{proof}

\begin{lem}\label{lem:RLODdecompnew}
    Let $2\leq j\leq r$. For any integers $a$ and $b$ that satisfy the inequalities given in \eqref{eq:abjineq}, the rational function
\begin{equation}\label{eq:PLabj}
\frac{P_L\bigl(O_w(a,b,j)\to D_w(a,b,j)\bigr)}{\prod\limits_{1_\mu\in D_w(a,b,j)}\rho(1_{\mu})}\end{equation} 
has the form
  \begin{equation} \label{RLnform}
  \begin{gathered}
    \sum\limits_{{\mathcal V}\subseteq V_w\left(a,b,j\right)} c({\mathcal V})f_{\mathcal V}, \quad
\text{where} \quad f_{\mathcal V}\in \Z\bigl[\{u_{x,y} \mid n_{x,y}\in {\mathcal V}\}\bigr] \quad \text{and} \\ \quad c({\mathcal V})\in \Z\bigl[\{n_{\pi(x),r}\in V_w\mid a<x\leq b\}\bigr].
\end{gathered}
\end{equation}
\end{lem}
\begin{proof}
We prove the lemma by descending induction on $j$. First, let $j=r$. For $a$ and $b$ to satisfy \eqref{eq:abjineq}, we must have $a=\pi\i(r)$, $\pi\i(r)<b$, and $\pi(b)<r$. It follows that $O_w(a,b,j)=\{n_{\pi(b),r}\}$ and $D_w(a,b,j)=\{1_r\}$. Since the path from $n_{\pi(b),r}$ to $1_r$ is vertical, it follows from Definition~\ref{calpleft} that the rational function \eqref{eq:PLabj} equals $0$. The lemma follows. \par
Now, let $2\leq j\leq r-1$, and assume that the assertion is true for all $j'$, $j<j'\leq r$. Let $a$ and $b$ be positive integers satisfying the inequalities \eqref{eq:abjineq}.
Recall from the remark below Definition~\ref{def:DabjOabj} that $1_{\pi(a)}$ is the highest element of $D_w(a,b,j)$. Write
\[
    D_w(a,b,j)=\{1_{l(1)},1_{l(2)},\ldots,1_{l(m)}\},\quad l(1)<l(2)<\cdots<l(m).
\]
Let $1\leq h\leq m$ satisfy $l(h)=\pi(a)$. Since $\pi\i(l(1))<b$ and $\pi(b)<j\leq l(1)$, we have $n_{\pi(b),l(1)}\in V_w$. First, we claim that $R_L(n_{\pi(b),l(1)})$ is of the form \eqref{RLnform}. By Lemma 4.1 of \cite{kim2024pt1}, $R_L(n_{\pi(b).l})$ equals the sum
\[u_{\pi(b),l(1)}-n_{\pi(b),r}u_{\pi(b-1),l(1)}-R_1(n_{\pi(b),l(1)}).\]
Check that the first and second terms are of the form \eqref{RLnform}. Observe that any linear combination of products of functions of the form \eqref{RLnform} is again of the form \eqref{RLnform}. Therefore, it follows immediately from the induction hypothesis and Lemma~\ref{lem:R1ndecomp} that $R_1(n_{\pi(b),l(1)})$ is also of the form \eqref{RLnform}. This proves our claim. \par
We now proceed to analyze the rational function \eqref{eq:PLabj}. If $h=1$, then $1_{\pi(a)}=1_{l(1)}$ is the leftmost element of $D_w(a,b,j)$. In this case, it follows from \eqref{eq:Dabj} and \eqref{eq:Oabj} that $D_w(a,b,j)=D_w(a,b,\pi(a))$ and $O_w(a,b,j)=O_w(a,b,\pi(a))$. By Lemma~\ref{lem:nxyabjform}, we deduce that \eqref{eq:PLabj} equals $R_L(n_{\pi(b),\pi(a)})$ up to a sign, which we already know has the form \eqref{RLnform}. Henceforth, assume that $h>1$. By Lemma 5.15 of \cite{kim2024pt1}, we have
\begin{equation}\label{PLODabj1}
\begin{aligned}
    P_L(O_w(a,b,j)&\to D_w(a,b,j))=P_L(O_{a,b,j,\downarrow}\to D_{a,b,j,\downarrow})\cdot P_L(O_{a,b,j,\uparrow}\to D_{a,b,j,\uparrow}) \\
    & -\rho(1_{l(1)})\cdot P_L(O_w(a,b,j)\smallsetminus \{n_{l(1),r}\} \to D_w(a,b,j)\smallsetminus \{1_{l(1)}\}),
\end{aligned}
\end{equation}
where
\begin{align*}
    O_{a,b,j,\downarrow}&=\bigl\{\gamma(1_{l(k)})\in O_w(a,b,j) \mid \pi\i\bigl(l(k)\bigr)\geq \pi\i\bigl(l(1)\bigr)+1\bigr\}, \\
    D_{a,b,j,\downarrow}&=\bigl\{1_{l(k)}\in D_w(a,b,j) \mid \pi\i\bigl(l(k)\bigr)\geq \pi\i\bigl(l(1)\bigr)\bigr\},  \\
    O_{a,b,j,\uparrow}&=\bigl\{\gamma(1_{l(k)})\in O_w(a,b,j) \mid \pi\i\bigl(l(k)\bigr)\leq \pi\i\bigl(l(1)\bigr)\bigr\},  \\
    D_{a,b,j,\uparrow}&=\bigl\{1_{l(k)}\in D_w(a,b,j) \mid \pi\i\bigl(l(k)\bigr)\leq \pi\i\bigl(l(1)\bigr)-1\bigr\}. 
\end{align*}
Since $ D_{a,b,j,\downarrow}$ and  $D_{a,b,j,\uparrow}$ partition $D_w(a,b,j)$, we have
\[
\prod\limits_{1_\mu\in D_w(a,b,j)} \rho(1_\mu)=\prod\limits_{1_\mu\in D_{a,b,j,\downarrow}} \rho(1_\mu)\cdot \prod\limits_{1_\mu\in D_{a,b,j,\uparrow}} \rho(1_\mu)=\rho\left(1_{l(1)}\right)\cdot \prod\limits_{1_\mu\in D_w(a,b,j)\smallsetminus \{1_{l(1)}\}} \rho(1_\mu).\]
Thus, dividing both sides of \eqref{PLODabj1} by $\prod\limits_{1_\mu\in D_w(a,b,j)} \rho(1_\mu)$, we obtain
\begin{equation}\label{PLODabj2}
\begin{aligned}
    \frac{P_L\left(O_w(a,b,j)\to D_w(a,b,j)\right)}{\prod\limits_{1_\mu\in D_w(a,b,j)} \rho(1_\mu)}&=\frac{P_L(O_{a,b,j,\downarrow}\to D_{a,b,j,\downarrow})}{\prod\limits_{1_\mu\in D_{a,b,j,\downarrow}} \rho(1_\mu)}\cdot \frac{P_L(O_{a,b,j,\uparrow}\to D_{a,b,j,\uparrow})}{\prod\limits_{1_\mu\in D_{a,b,j,\uparrow}} \rho(1_\mu)} \\
    &\quad -\frac{P_L(O_w(a,b,j)\smallsetminus \{n_{l(1),r}\} \to D_w(a,b,j)\smallsetminus \{1_{l(1)}\})}{\prod\limits_{1_\mu\in D_w(a,b,j)\smallsetminus \{1_{l(1)}\}} \rho(1_\mu)}.
\end{aligned}
\end{equation}
Since every element of $D_{a,b,j,\downarrow}$ is in or below the $\pi\i(l(1))$-th row and in or to the right of the $l(1)$-th column, it follows from Definition~\ref{def:DabjOabj} and Lemma~\ref{lem:nxyabjform} that
\[
D_{a,b,j,\downarrow}=D_w\bigl(b,\pi\i(l(1)),l(1)\bigr)=D\bigl(n_{\pi(b),l(1)}\bigr).\]
From this, it immediately follows that
$O_{a,b,j,\downarrow}=O\bigl(n_{\pi(b),l(1)}\bigr)$. Hence the first fractional term on the right hand side of \eqref{PLODabj2} equals $R_L(n_{\pi(b),l(1)})$ up to a sign, which has the form \eqref{RLnform}. Similarly, it is straightforward to verify that \[D_{a,b,j,\uparrow}=D_w\bigl(a,\pi\i(l(1)),l(2)\bigr),\quad  O_{a,b,j,\uparrow}=O_w\bigl(a,\pi\i(l(1)),l(2)\bigr),\]
\[D_w(a,b,j)\smallsetminus \{1_{l(1)}\}=D_w\left(a,b,l(2)\right), \quad \text{and}\quad O_w(a,b,j)\smallsetminus \{n_{l(1),r}\}=O_w\bigl(a,b,l(2)\bigr).\]
Since $j\leq l(1)<l(2)$, it follows from the induction hypothesis that the second and the third fractional terms on the right hand side of \eqref{PLODabj2} are of the form \eqref{RLnform}. This finishes the proof.
\end{proof}
\begin{cor}\label{lem:RLndecomp}
If $n_{\pi(d),j}\in V_w$ then $R_L(n_{\pi(d),j})$ can be written in the form
\begin{equation} \label{RLnform1}
\begin{gathered}
    \sum\limits_{{\mathcal V}\subseteq V_w(\pi\i(j),d,j)} c({\mathcal V})f_{\mathcal V}, \quad \text{where} \quad  f_{\mathcal V}\in \Z\bigl[\bigl\{u_{x,y} \mid n_{x,y}\in {\mathcal V}\bigr\}\bigr] \quad \text{and} \\ 
    c({\mathcal V})\in \Z\bigl[\bigl\{n_{\pi(k),r}\in V_w\mid \pi\i(j)<k\leq d\bigr\}\bigr].
\end{gathered}
\end{equation}
\end{cor}
\begin{proof}
This follows immediately from Lemma~\ref{lem:nxyabjform} and the previous lemma.
\end{proof}
\subsection{Proof of part (v)}
We now move on to analyzing the rational functions $n_\a\f{\partial }{\partial n_\a} u_\b$. Recall from Lemma~\ref{lem:basecase} that $u_{1,2}=n_{1,2}$, so part (v) is trivial for $\GL(2)$. We use induction on $r$, assuming that part (v) holds for every Weyl group element $w'$ of $\GL(r')$ with $2\leq r'<r$. The proof proceeds by a case-by-case analysis.
\begin{lem}\label{lem:p6abovebottom}
Suppose that $n_\b\in V_w$ is above the bottom row of $wn$. For any $n_\a\in V_w$, $n_\a\f{\partial }{\partial n_\a} u_\b$ is of the form \eqref{ordform},
\[\sum\limits_{i} c_i(\a)f_i, \quad \text{where}\quad c_i(\a)\in \Z\bigl[\{n_\d,n_\d\i\mid \d\prec\a\}\bigr]\quad\text{and}\quad f_i\in \Z\bigl[\{u_\g \mid n_\g \in V_w \} \bigr].\]
\end{lem}
\begin{proof}
If $n_\a$ is in the bottom row, then $n_\a$ does not appear in the rational function $u_\b=R(n_\b)$. It follows that $n_\a\f{\partial }{\partial n_\a} u_\b=0$, and the lemma follows trivially. Now, assume that both $n_\b$ and $n_\a$ are above the bottom row. As before in Section~\ref{sec:parts23case1}, let $\what$ be the matrix obtained from $w$ by removing the bottom row and the $\pi(r)$-th column. Recall the definition \eqref{phihat} of the map $\phihat$. For a free variable $n_{a,b}=n_\g\in V_w$ above the bottom row of $wn$, we understand $\phihat(\g)$ as $\bigl(\phihat(a),\phihat(b)\bigr)\in \Inv(\pihat\i)$. Since $n_\a$ and $n_\b$ are both above the bottom row of $wn$, it follows from Lemma~\ref{indiceslem} that $\phihat (n_\a)$ and $\phihat(n_\b)$ are free variables in $V_{\what}$. By the induction hypothesis, the rational function
$\phihat(n_\a) \f{\partial }{\partial (\phihat(n_\a))} R^{(\what)}\bigl(\phihat(n_\b)\bigr)$
is of the form $\sum\limits_{i} c_i\bigl(\phihat(\a)\bigr)f_i$, with
\[c_i\bigl(\phihat(\a)\bigr)\in \Z\left[\{n_\d,n_\d\i\mid \d\in \Inv(\pihat\i),\, \d\prec\phihat(\a)\}\right] \quad \text{and}\quad f_i\in \Z\left[\{u_\g \mid n_\g \in V_{\what}\} \right].\]
Since the map $\phihat$ simply changes the indices, it follows from Lemma~\ref{indiceslem} that we have
\[
     n_\a  \f{\partial}{\partial n_\a} R^{(w)}(n_\b)=\phihat\i\Bigl(\phihat(n_\a) \f{\partial R^{(\what)}\bigl(\phihat(n_\b)\bigr)}{\partial \bigl(\phihat(n_\a)\bigr)} \Bigr).
\]
Therefore, to prove the lemma, it suffices to show that the rational function
\begin{equation}\label{phihatinvpartial}
   \phihat\i\Bigl(\sum\limits_{i} c_i\bigl(\phihat(\a)\bigr)f_i\Bigr)=\sum\limits_{i} \phihat\i\left(c_i\bigl(\phihat(\a)\bigr)\right)\phihat\i(f_i)
\end{equation}
has the form \eqref{ordform}. By Lemma~\ref{indiceslem}, we have $ \phihat\i(f_i)\in \Z\bigl[\{u_\g \mid n_\g \in V_w\}\bigr]$. Recall the ordering \eqref{ordering}. With this ordering, it is straightforward to verify that if $n_{\g_1},n_{\g_2}\in V_w$ are above the bottom row of $wn$, then $\phihat(n_{\g_1})\succ \phihat(n_{\g_2})$ in the ordering of $V_{\what}$ if and only if $n_{\g_1} \succ n_{\g_2}$ in the ordering of $V_w$. We conclude that \eqref{phihatinvpartial} has the form \eqref{ordform}.
\end{proof}

\begin{lem}\label{lem:p6rr}
Suppose that the bottom right corner entry $n_{\pi(r),r}$ of $wn$ is a free variable. Then for any $n_\a\in V_w$, $n_\a \frac{\partial}{\partial n_\a}u_{\pi(r),r}$ has the form \eqref{ordform}.
\end{lem}

\begin{proof}
It follows immediately from Lemma~\ref{lem:Rrm} that $n_\a \frac{\partial}{\partial n_\a}u_{\pi(r),r}$ equals $u_{\pi(r),r}$ if $n_\a$ is in the rightmost column of $wn$ and zero otherwise. The lemma follows.
\end{proof}

\begin{lem}\label{lem:p6naorigin}
Let $n_{\b}$ be a free variable in the bottom row and $n_\b\neq n_{\pi(r),r}$. If $n_\a\in V_w$ is an element of the set $O(n_\b)$ of origins of $n_\beta$, then $n_\a\f{\partial}{\partial n_\a} u_\b$ has the form \eqref{ordform}.
\end{lem}

\begin{proof}
Write $n_\b=n_{\pi(r),j}$ with $j<r$. By \eqref{defR}, we have
     \[
u_\b=(-1)^{t_{\pi(r),j}}\frac{P(\nprj)}{\prod\limits_{1_\mu\in D(\nprj)}\rho(1_\mu)},\quad\text{where} \quad t_{\pi(r),j}=\abs{D(\nprj)}-1.\]
If $n_\a$ is an origin of $\nprj$, then $n_\a$ is in every set of disjoint paths in ${\mathcal P}(\nprj)$, hence $n_\a$ is in every term of the polynomial $P(\nprj)$. If $n_\a$ is the bottom origin $n_{\pi(r),r}$ of $\nprj$, then $n_\a$ is below every destination of $\nprj$, hence $n_\a$ is not in the denominator. It follows that $n_\a\f{\partial }{\partial n_\a} u_{\b}$ equals $u_{\b}$,
which is of the form \eqref{ordform}. On the other hand, if $n_\a\neq n_{\pi(r),r}$, then by Definition~\ref{Odef}, $n_\a$ is in the denominator of $R(\nprj)$. In this case, $n_\a$ is canceled out and does not appear in $u_\b$, thus $n_\a\f{\partial }{\partial n_\a} u_\b=0$. The lemma follows.
\end{proof}
It remains to consider the case where $n_\b$ is in the bottom row and $n_\a\notin O(n_\b)$.
\begin{lem}\label{lem:p6RL}
Let $n_{\b}=\nprj\in V_w$ with $j<r$. Suppose that $n_\a\in V_w$ and $n_\a \notin O(n_\b)$. Then $n_\a \frac{\partial}{\partial n_\a} R_L(n_\b)$ has the form \eqref{ordform}.
\end{lem}
\begin{proof}
By Definition~\ref{def:RLR1R2}, we have
\[
R_L(n_\b)=(-1)^{t_{\pi(r),j}}\frac{P_L(\nprj)}{\prod\limits_{1_\mu\in D(\nprj)}\rho(1_\mu)},\quad\text{where} \quad t_{\pi(r),j}=\abs{D(\nprj)}-1.\]
First, consider the case that $n_\a$ is in the rightmost column of $wn$. It follows from Definitions~\ref{calpleft} and~\ref{def:calpleftn} that $n_\a$ does not appear in the polynomial $P_L(\nprj)$. Also, $n_\a$ appears in the denominator of $R_L(\nprj)$ only if $n_\a$ is the rightmost nonzero entry of the top destination $1_j$ of $n_{\pi(r),j}$, that is, $n_\a=n_{j,r}\in V_w$. Otherwise, $n_\a$ does not appear in the rational function $R_L(n_\b)$. It follows that
\begin{equation}\label{eq:p6RLrm}
n_\a \frac{\partial}{\partial n_\a}R_L(n_\b)=\begin{cases}
    -R_L(n_\b) & \text{if } n_\a=n_{j,r} \\
    0 & \text{otherwise.}
\end{cases}\end{equation}
By Corollary~\ref{lem:RLndecomp},
$R_L(n_\b)$ has the form $\sum\limits_{{\mathcal V}\subseteq V_w(\pi\i(j),r,j)} c({\mathcal V})f_{\mathcal V}$, where
\[ f_{\mathcal V}\in \Z\bigl[\bigl\{u_{x,y} \mid n_{x,y}\in {\mathcal V}\bigr\}\bigr] \eqand
    c({\mathcal V})\in \Z\bigl[\bigl\{n_{\pi(k),r}\in V_w\mid \pi\i(j)<k\leq r\bigr\}\bigr].
\]
If $n_\a=n_{j,r}$, then $n_\a\succ n_{\pi(k),r}$ for any $\pi\i(j)<k\leq r$, thus $ c({\mathcal V})\in \Z\left[\{n_\d,n_\d\i\mid \d\prec\a\}\right]$. We conclude that $n_\a\f{\partial }{\partial n_\a} R_L(n_\b)$ has the form \eqref{ordform}. \par
Now, assume that $n_\a$ is not in the rightmost column of $wn$. As before in Section~\ref{sec:parts23case2}, let $\wtilde$ be the matrix obtained from $w$ by removing the rightmost column and the $\pi\i(r)$-th row. Recall the definition \eqref{phitilde} $\phitilde$ and Definition~\ref{def:wtildenu} of $\wtilde\ntilde$ and $\wtilde\utilde$. By Lemma~\ref{wntildewtilden}, $n_\a$ and $n_\b$ are elements of $V_{\wtilde}$. By Lemma~\ref{RLwRwtilde}, we have
\[
R^{(w)}_L(n_\b)=\begin{cases}
-n_{\pi(r),r}\cdot R^{(\widetilde{w})}(n_\b) & \text{if} \quad \pi\i(j)<\pi\i(r), \\
\frac{n_{\pi(r),r}}{n_{j,r}}\cdot R^{(\widetilde{w})}(n_\b) & \text{if} \quad \pi\i(j)>\pi\i(r).
\end{cases}
\]
For convenience, write $R^{(w)}_L(n_\b)=s(\b) R^{(\widetilde{w})}(n_\b)$, where $s(\b)$ is either $-n_{\pi(r),r}$ or $\frac{n_{\pi(r),r}}{n_{j,r}}$.
Since $n_\a$ is not in the rightmost column, we have 
$n_\a\f{\partial }{\partial n_\a} R_L^{(w)}(n_{\b})=s(\b) n_\a\f{\partial }{\partial n_\a} R^{(\widetilde{w})}(n_{\b})$. Regarding $n_\a$ as a free variable in $V_{\wtilde}$, we apply the induction hypothesis for part (v) of Theorem~\ref{biratlthm} and write
\[
n_\a\f{\partial }{\partial n_\a} R_L^{(w)}(n_{\b})=s(\b) n_\a\f{\partial }{\partial n_\a} R^{(\widetilde{w})}(n_{\b})=s(\b)\Bigl( \sum\limits_{i} c_i(\a)f_i\Bigr),\]
where $ c_i(\a)\in \Z[\{n_\d,n_\d\i\mid \d\in \Inv(\pitilde\i),\ \d\prec\a\}]$ and $f_i\in \Z[\{u_\g \mid n_\g \in V_{\wtilde}\} ]$. Recall the ordering \eqref{ordering}. Check that if $n_{\g_1},n_{\g_2}\in V_w$ are not in the rightmost column of $wn$, then $n_{\g_1}\succ n_{\g_2}$ in the ordering of $V_w$ if and only if $n_{\g_1}\succ n_{\g_2}$ in the ordering of $V_{\wtilde}$. Therefore, by Lemma~\ref{wntildewtilden}, we deduce
$c_i(\alpha)\in \Z[\{n_\d,n_\d\i \mid \d\in \Inv(\pi\i),\ \d\prec\a \}]$.
Also, since $n_\a \succ n_\g$ for any $n_\g\in V_w$ in the rightmost column, we have $s(\b) \in \Z\bigl[\{n_\d,n_\d\i\mid \d\in \Inv(\pi\i),\ \d\prec\a\}\bigr]$. Finally, by Lemma~\ref{wntildewtilden}, we have
$f_i\in \Z\bigl[\{u_\g \mid n_\g \in V_{w}\} \bigr]$.
We conclude that $n_\a\f{\partial }{\partial n_\a} R_L(n_\b)$ is of the form \eqref{ordform}. This finishes the proof.
\end{proof}
We are ready to finish the proof of part (v) of Theorem~\ref{biratlthm}.
\begin{proof}[Proof of part (v) of Theorem~\ref{biratlthm}]
By Lemma~\ref{lem:p6abovebottom} and~\ref{lem:p6rr}, it suffices to consider the case $n_\b=n_{\pi(r),j}$ with $j<r$. With Lemma~\ref{lem:p6rr} being the base case, we use descending induction on the column index $j$, assuming that part (v) holds for any $n_{\b'}=n_{\pi(r),j'}$ with $j<j'\leq r$. By Lemma~\ref{lem:p6naorigin}, we may assume $n_\a\notin O(\nprj)$.
By Lemma 4.1 of \cite{kim2024pt1}, we have
\[
    R(\nprj)=n_{\pi(r),r}\cdot u_{\pi(r-1),j}+R_L(\nprj)+R_1(\nprj).\]
Observe that any product or any $\Z$-linear combination of rational functions of the form \eqref{ordform} is again of the form \eqref{ordform}, thus we may consider the three terms separately. Since $n_{\pi(r),r}$ is the bottom origin of $\nprj$, we have $n_\a\neq n_{\pi(r),r}$, consequently $n_\a\succ n_{\pi(r),r}$. From this and Lemma~\ref{lem:p6abovebottom}, we deduce that the function $n_{\pi(r),r}\cdot u_{\pi(r-1),j}$ after applying the differential operator $n_\a \frac{\partial}{\partial n_\a}$ has the form \eqref{ordform}. Also, by Lemma~\ref{lem:p6RL}, $n_\a \frac{\partial}{\partial n_\a}R_L(\nprj)$ has the form \eqref{ordform}. It remains to show that $n_\a \frac{\partial}{\partial n_\a}R_1(\nprj)$ has the form \eqref{ordform}. Write
$O_1(\nprj)=\{n_{\pi(\d_1),r},\ldots,n_{\pi(\d_s),r}\}$.
By Lemma~\ref{lem:R1ndecomp}, it suffices to show that, for each $1\leq m\leq s$,
\begin{equation}\label{R1terms}
   n_\a \frac{\partial}{\partial n_\a} \Bigl[ n_{\pi(\delta_m),r}\cdot u_{\pi(\d_m-1),j} \cdot \frac{P_L(O_w(\d_m,r,j+1)\to D_w(\d_m,r,j+1))}{\prod\limits_{1_\mu\in D_w(\d_m,r,j+1)}\rho(1_\mu)}\Bigr]
\end{equation}
has the form \eqref{ordform}. Observe that we have $n_{\a}\neq n_{\pi(\d_m),r}$ because $n_\a \notin O(\nprj)$.
\par 
First, consider the case $n_{\pi(\d_m),r}\succ n_\a$. This implies that $n_\a$ is in the rightmost column and below the $\d_m$-th row of $wn$, hence does not appear in the rational function $u_{\pi(\d_m-1),j}$. Also, by Lemma~\ref{lem:nxyabjform} we have $O_w(\d_m,r,j+1)\subseteq O(\nprj)$ and $D_w(\d_m,r,j+1)\subseteq D(\nprj)$. By the same reasoning used for \eqref{eq:p6RLrm}, we deduce that $n_\a$ does not appear in the rational function 
\begin{equation}\label{R1fraction}
\frac{P_L(O_w(\d_m,r,j+1)\to D_w(\d_m,r,j+1))}{\prod\limits_{1_\mu\in D_w(\d_m,r,j+1)}\rho(1_\mu)}.\end{equation}
We conclude that \eqref{R1terms} equals $0$.
\par
It remains to consider the case $n_\a\succ n_{\pi(\d_m),r}$. By Leibniz's rule, to prove that \eqref{R1terms} is of the form \eqref{ordform}, it suffices to show that $n_{\pi(\delta_m),r} u_{\pi(\d_m-1),j}$, $n_\a \frac{\partial}{\partial n_\a}\left(n_{\pi(\delta_m),r}\cdot u_{\pi(\d_m-1),j}\right)$, the rational function \eqref{R1fraction}, and \eqref{R1fraction} after applying $n_\a \frac{\partial}{\partial n_{\a}}$ are of the form \eqref{ordform}. It follows directly from the assumption $n_\a\succ n_{\pi(\d_m),r}$ and Lemma~\ref{lem:p6abovebottom} that the first two functions are of the form \eqref{ordform}. Also, by Lemma~\ref{lem:RLODdecompnew}, the function \eqref{R1fraction} has the form \[\sum\limits_{{\mathcal V}\subseteq V_w(\d_m,r,j+1)} c({\mathcal V})f_{\mathcal V},\] where $f_{\mathcal V}\in \Z\bigl[\{u_{x,y} \mid n_{x,y}\in {\mathcal V}\}\bigr]$ and $c({\mathcal V})\in \Z\bigl[\{n_{\pi(x),r}\in V_w \mid \d_m<x\leq r\}\bigr]$. It follows from the assumption $n_\a\succ n_{\pi(\d_m),r}$ that $n_\a\succ n_{\pi(x),r}$ whenever $\d_m<x\leq r$. Therefore, \eqref{R1fraction} is of the form \eqref{ordform}. Finally, to show that \eqref{R1fraction} after applying $n_\a \frac{\partial}{\partial n_\a}$ is of the form \eqref{ordform}, it suffices to show that for any ${\mathcal V}\subseteq V_w(\d_m,r,j+1)$ the function $n_\a \frac{\partial}{\partial n_\a} \bigl(c({\mathcal V})f_{\mathcal V}\bigr)$ is of the form \eqref{ordform}. Since $n_\a \succ n_{\pi(x),r}$ for any $\d_m<x\leq r$, we have $c({\mathcal V}) \in \Z[\{n_\d,n_\d\i\mid \d\prec\a\}]$ and
\[n_\a \frac{\partial}{\partial n_\a} \left(c({\mathcal V})f_{\mathcal V}\right)=c({\mathcal V}) n_\a \frac{\partial}{\partial n_\a} f_{\mathcal V}.\]
Also, by Definition~\ref{def:Vwabj}, every free variable $n_{x,y}$ in the set $V_w(\d_m,r,j+1)$ is to the right of the $j$-th column. It follows from Lemma~\ref{lem:p6abovebottom} and the induction hypothesis on the column index that $n_\a \frac{\partial}{\partial n_\a} f_{\mathcal V}$ has the form \eqref{ordform}. This finishes the proof.
\end{proof}

\section{Proof of Lemma~\ref{lem:intdom}}\label{sec:intdom}
The last missing piece for the analytic continuation of Jacquet integrals is Lemma~\ref{lem:intdom}. Let the ordering $\sqsupset$ be as defined in \eqref{ordering2}, and write
\[V_w=\{n_{\b_1},n_{\b_2},\ldots, n_{\b_d}\}, \quad n_{\beta_1}\sqsupset n_{\beta_2}\sqsupset \cdots \sqsupset n_{\beta_d}.\]
According to this ordering, $n_{\b_d}$ is in the rightmost column that contains a free variable, and is in the highest row among the free variables in that column (in the example \eqref{ex:w3}, $n_{4,5}$ is the last element in the ordering). Such $n_{\b_d}$ is right below the entry of $1$ in that column, hence $n_{\b_d}=u_{\b_d}$. Recall from Proposition~\ref{lem:Jacobian} that $\partial u_{\b_i}/\partial n_{\b_j}=0$ whenever $i>j$. Hence $u_{\b_i}$ is a rational function of $n_{\b_i}, n_{\b_{i+1}}\ldots, n_{\b_d}$. For each $j$ with $1\leq j\leq d$, let $t_j= \abs{D(n_{\b_j})}-1$,
\[f_j=(-1)^{t_j}\Bigl(u_{\b_j}-\frac{\partial u_{\b_j}}{\partial n_{\b_j}}n_{\b_j}\Bigr),\]
and
\[h_j=(-1)^{t_j}\frac{M+f_j}{\partial u_{\b_j}/\partial n_{\b_j}}.\]
The sign factor $(-1)^{t_j}$ ensures that $f_j(\abs{n_{\b_{j+1}}},\ldots, \abs{n_{\b_d}})$ and $h_j(\abs{n_{\b_{j+1}}},\ldots, \abs{n_{\b_d}})$ are both nonnegative. Check that $f_d\equiv 0$ and $h_j\equiv M$, hence $h_j$ satisfies part (i) of the lemma. Next, recall from the proof of Lemma~\ref{lem:part4partial} that
\[\frac{\partial u_{\b_j}}{\partial n_{\b_j}}=(-1)^{t_j}\frac{1}{n_{\b_j}}\frac{u(\mathbf{p})}{\prod\limits_{1_\mu\in D(n_{\b_j})}\rho(1_\mu)},\]
where $\mathbf{p}$ is the unique set of disjoint paths containing $n_{\b_j}$. From this, it is straightforward to check that $h_j$ satisfies part (ii) of the lemma. Finally, for part (iii), observe that
\[u_{\b_j}=\frac{\partial u_{\b_j}}{\partial n_{\b_j}}n_{\b_j}+(-1)^{t_j}f_j.\]
We deduce that $n_{\b_j}\in I_j$ whenever $u_{\b_j}\in [-M,M]$. This finishes the proof.

\bibliographystyle{plain}
\bibliography{WhitDistref}
\end{document}